\documentclass[10pt]{article}
\usepackage[utf8]{inputenc}
\usepackage[a4paper,margin=1.2in]{geometry}
\usepackage{amsmath}
\usepackage{amsthm}
\usepackage{graphicx}
\usepackage{amssymb}
\usepackage{comment}
\usepackage{pgfplots}
\usepackage{subcaption}

\newtheorem{lem}{Lemma}

\newtheorem{rem}{Remark}
\newtheorem{cor}{Corollary}

\newtheorem{thm}{Theorem}

\def\w{\omega}
\def\p{\partial}
\DeclareMathOperator{\supp}{supp}

\title{Finite time blow-up for the hypodissipative Navier Stokes equations with a force in $L^1_t C_x^{1,\epsilon}\cap L^{\infty}_{t} L^{2}_{x}$}

\author{Diego C\'ordoba\footnote{dcg@icmat.es, Instituto de Ciencias Matem\'aticas CSIC-UAM-UCM-UC3M  },\quad Luis Mart\'inez-Zoroa\footnote{luis.martinezzoroa@unibas.ch, University of Basel}\quad and Fan Zheng \footnote{fan.zheng@icmat.es, Instituto de Ciencias Matem\'aticas CSIC-UAM-UCM-UC3M  }}

\date{}
\pgfplotsset{compat=newest}
\begin{document}

\maketitle
\begin{abstract}
    In this work we establish the formation of singularities of classical solutions with finite energy of the forced fractional Navier Stokes equations where the dissipative term is given by $|\nabla|^{\alpha}$ for any $\alpha\in [0, \alpha_0)$ ($\alpha_0 = \frac{22-8\sqrt7}{9} > 0$). We construct solutions in $\mathbb{R}^3\times [0,T]$ with a finite $T>0$ and  with an external forcing which is in $L^1_t([0, T]) C_x^{1,\epsilon}\cap L^{\infty}_{t} L^{2}_{x}$, such that on the time interval $0 \le t < T$, the velocity $u$ is in the space $C^\infty\cap L^2$ and such that as the time $t$ approaches the blow-up moment $T$, the integral $\int_0^t |\nabla u| ds$ tends to infinity.  
\end{abstract}

\section{Introduction}
We consider the initial value problem for the fractional Navier-Stokes equations  in $\mathbb{R}^{3} $ :
\begin{equation}\label{fNS}
\begin{aligned}
u_t+u\cdot\nabla u + \nabla P   + \nu |\nabla|^{\alpha} u &= f, \quad (x,t) \in \mathbb{R}^{3} \times \mathbb{R}_+ \\
\quad\quad\quad \nabla \cdot u &= 0
\end{aligned}
\end{equation}
with initial data $u_0= u(x,0)$, where
\begin{align*}
    u &=(u_1(x,t), u_2(x,t), u_3(x,t))\text{ is the velocity field},\\
    P&=P(x,t)\text{ is the pressure function},\\
    \nu&>0\text{ is the viscosity coefficient},\\
    f&=(f_1(x,t), f_2(x,t), f_3(x,t))\text{ is an external force}.
\end{align*}
 We let $|\nabla|^{\alpha} g\equiv (-\Delta)^{\frac{\alpha}{2}} g$ be the Fourier multiplier $\widehat{|\nabla|^{\alpha} g}(\xi) = |\xi|^{\alpha}\widehat{g} (\xi)$. 
 
When $\alpha=2$, the equations align with the well-recognized classical Navier-Stokes equations. If $0<\alpha<2$, these equations are called hypodissipative Navier-Stokes equations, indicating their reduced kinetic energy dissipation compared to the complete Laplacian $\Delta$. From a physical perspective, this model characterizes fluids that exhibit internal friction as described in \cite{MGSG} and has also been derived using a stochastic Lagrangian particle methodology according to \cite{Z}. On the other hand, if $\alpha>2$, the equations are considered hyperdissipative. When $\nu=0$ we have the incompressible Euler equations.

We can deduce their corresponding energy equality
$$\frac12\int_{\mathbb{R}^{3}} |u(x,t)|^2 dx + \nu\int_0^t\int_{\mathbb{R}^{3}} ||\nabla|^{\alpha/2} u(x,s)|^2 dxds   = \frac12\int_{\mathbb{R}^{3}} |u(x,0)|^2 dx +\int_0^t\int_{\mathbb{R}^{3}} u(x,s)f(x,s) dxds$$ for smooth solutions by multiplying the initial equation by u and performing integration by parts.
Also, by taking the curl of the equation we get its vorticity formulation
\begin{equation}\label{vorticity}
\omega_t+u\cdot\nabla\omega + \nu |\nabla|^{\alpha} \omega = \nabla \times f + \omega \cdot \nabla u
\end{equation}
where $u$ can now be recovered from $\omega$ using the Biot--Savart law
\[
u(x) = \omega * \mathcal K(x) = \frac1{4\pi}\int_{\mathbb R^3} \frac{\omega(y) \times (x - y)}{|x - y|^3}dy.
\]
where $\mathcal K$ is the Biot--Savart kernel.

The term $\omega \cdot \nabla u$ on the right-hand side means that in addition to being transported by the velocity,
the vorticity itself will be stretched by the deformation caused by the flow; hence its name ``the vorticity stretching term".
This term is peculiar to three dimensions: in two dimensions the vorticity always points perpendicularly to the flow,
so the vorticity stretching term vanishes. This explains the stark contrast in the global regularity theory in two and three dimensions:
In two dimensions global well-posedness is well known \cite{Hol,wol}, while in three dimensions the corresponding problem is the famous
 Millennium Prize Problems \cite{Fef}: If the initial velocity $u_0$ is smooth and the external force $f$ is uniformly smooth, does a smooth solution to these equations exist for all time, or is there a finite time singularity? This question is a cornerstone of mathematical fluid dynamics and remains unsolved, stimulating extensive research to shed light on the potential existence or non-existence of singularities.

Following the groundbreaking work by J. Leray \cite{Le}, a wealth of extensive literature has been developed on the existence theory for solutions of the Navier-Stokes equations. A significant advancement in regularity theory is the renowned partial-regularity theory developed by Caffarelli-Kohn-Nirenberg \cite{CKN}, who proved that the 1-dimensional parabolic Hausdorff measure of the singular set (the points $(x, t)$ where $u$ is unbounded in any neighborhood of $(x, t)$ is zero) for every appropriate weak solution. More recently,  this theory has been expanded in \cite{TY} to the hypodissipative range $\frac32 < \alpha < 2$ by proving that the $(5 - 2\alpha)$-dimensional Hausdorff measure of the singular set is zero for every fitting weak solution.  In the scenario of hyperdissipation, where $2 < \alpha \leq \frac52$, a parallel achievement has been recently shown in \cite{CDM} (see also  \cite{KP} and \cite{O}).

The balance of energy in the hyperdissipative Navier-Stokes equation reaches a critical scale at $\alpha = \frac52$. For this equation, global regularity is maintained for all solutions when $\alpha\geq\frac52$ (see the work of Lions \cite{L}), extending logarithmically into the range beyond the critical threshold (refer to \cite{BMR}, \cite{FFZ} and \cite{T1}). Additionally, as demonstrated in \cite{CH}, given any specific initial condition, $u_0$, belonging to the space $H^{\delta}$ with $\delta > 0$, such initial condition $u_0$ ensures the emergence of a globally smooth solution whenever $\alpha$ exceeds $\frac52 - \epsilon$, with $\epsilon$ being a small positive number determined solely by the $H^{\delta}$ magnitude of the initial condition.

Three significant results concerning the propagation of regularity for Navier-Stokes equations are the Prodi-Serrin criteria on the integrability of the velocity (see also \cite{ESS} and reference \cite{RR} for details and reviews on the criteria), the criterion on the direction field of the vorticity by Constantin-Fefferman in \cite{CF} and the Seregin-Sverak criteria in \cite{SS} on the lower bound of the pressure. These have been extensively extended in various studies to the fractional Navier-Stokes equation (see \cite{BG}, \cite{Ch}, \cite{W} and \cite{Kim} for further details and references therein). Conversely, there exist models of the fractional and classical Navier-Stokes equations that develop finite time singularities, in particular see \cite{T2} by Tao for a modified averaged Navier–Stokes system and  for a hypodissipative dyadic model see both \cite{KP2} by Katz-Pavlovic and \cite{Ches} by Cheskidov (see also \cite{ChDF} and \cite{EM}). Very recently a potential candidate for a finite time singularity for the classical Navier-Stokes equations has been presented in \cite{Hou5} by Hou.

The goal of this paper is to construct classical solutions in the well-posed regime of the hypodissipative Navier-Stokes equations that loses regularity in finite time. The main theorem is:     

\begin{thm}\label{thm}
There is $\alpha_0 = \frac{22-8\sqrt7}{9} > 0$ such that for any $\alpha\in [0, \alpha_0)$, there exist  solutions of the forced 3D incompressible fractional Navier-Stokes equations \eqref{fNS} in $\mathbb{R}^3\times [0, 1]$ with $\nu > 0$
\footnote{We will actually show the theorem for $\nu = 0.01$, but a rescaling argument shows the result for all $\nu > 0$.}
and an external forcing which is in $L^1_t([0, 1]) C_x^{1,\epsilon}\cap L^{\infty}_{t}L^{2}_{x}$ for some $\epsilon>0$, such that the velocity $u \in L_{t,x}^\infty \cap L_t^\infty L_x^2$,
and that on the time interval $0 \le t < 1$, the velocity $u$ is  smooth in $x$ and satisfies $$\lim_{t\rightarrow 1}\int_0^t |\nabla u (x,s)|_{L^{\infty}} ds = \infty.$$ 
\end{thm}

\begin{rem}
    Note that, using standard methods, there is local existence for hypodissipative Navier-Stokes in the spaces we are considering (see for example \cite{D}).
\end{rem}

\subsection{Previous results on finite time singularities for 3D incompressible Euler equations.}
Recently different scenarios of blow-up has been constructed within local well-posed regime ($C^{1,\beta}$) in the inviscid case ($\nu=0$): the 3D incompressible Euler equations. More precisely, there are 4 known scenarios of blow-up with finite energy and no boundaries:
\begin{itemize}
    \item Elgindi pioneered the first demonstration of finite-time singularities in Euler equations through his research \cite{Elgindi2}. He analyzed flows with axi-symmetric properties and no swirl, focusing on velocity profiles within the $C^{1,\beta }$ H\"older space, where $\beta$ is chosen to be a very small number, and which are $C^{\infty}$ outside of a zero measure set. His findings revealed the possibility of exact self-similar blow-up solutions that possess infinite energy and are not influenced by external forces. It's important to highlight that by applying a consistent non-zero $C^{1,\beta}$ external force, one can also derive solutions that blow up yet have finite energy, as Elgindi mentions in Remark 1.4 of his paper \cite{Elgindi2}. In subsequent research \cite{Elgindi3}, Elgindi, along with Ghoul and Masmoudi, expanded on these results, showing that such self-similar blow-up solutions could be localized. By examining the stability of these solutions, they confirmed the emergence of finite-time singularities in the Euler equations without any external forces ($f = 0$), specifically for solutions with finite energy in the $C^{1,\beta}$ framework. Recent progress has led to significant advancements and computer assisted proofs in the study of self-similar singularities in axi-symmetric flows when boundaries are present (refer to  \cite{EJ},  \cite{Hou},  \cite{Hou2} and \cite{Hou4}).
    \item In our recent work \cite{CMZ}, we introduced a unique blowup mechanism in $\mathbb{R}^3$  that diverges from the typical self-similar profile. We have constructed solutions for the 3D unforced incompressible Euler equations over the domain $\mathbb{R}^3\times [-T,0]$, where the velocity profiles lie within the function space $C^{\infty}(\mathbb{R}^3 \setminus {0})\cap C^{1,\beta}\cap L^2$, where $0<\beta\ll\frac13$, for time intervals $t\in (-T,0)$. These solutions are smooth everywhere except at a singular point and develop finite-time singularities at $t=0$ (see also \cite{Che2}, where a self-similar singularity smooth outside of a point is constructed). Remarkably, while these solutions are axi-symmetric without swirl, they do not rely on self-similar profiles. They are made up instead of a sequence of vorticity regions increasingly closer to the origin, creating a hyperbolic saddle at the center. This configuration induces motion and deformation, affecting the inner vortices. Thus, the Euler equations' dynamics under this scheme can be approximated by a solvable infinite system of ordinary differential equations, offering a precise backward-in-time analysis of the dynamics leading to the blow-up, thereby enabling a detailed proof of this singular behavior.
    \item In the recent study \cite{CM}, we introduce a blow-up mechanism for the forced 3D incompressible Euler equations, specifically targeting non-axisymmetric solutions. We develop solutions in $\mathbb{R}^3$ within the function space $C^{3,\frac{1}{2}}\cap L^2$ over the time interval $[0, T)$, where $T > 0$ is a finite value, and these solutions are subject to a continuous force in $C^{1,\frac{1}{2} -\epsilon}\cap L^2$. Notably, this approach leads to a blow-up: as time $t$ nears the critical moment $T$, the integral $\int_0^t |\nabla u| ds$ diverges, indicating infinite growth, yet the solution maintains smoothness everywhere except at the origin. Our blow-up construction does not employ self-similar coordinates and successfully addresses solutions beyond the $C^{1,\frac{1}{3}+}$ regularity, which is critical for axi-symmetric solutions without swirl.
    \item Very recently, Elgindi and Pasqualotto consider the setting of axi-symmetric flows with swirl. They proved in \cite{EP} a  self-similar profile blow-up for solutions in $C^{1, \beta}$  where the singularity happens away from the origin. The solutions considered actually live in $H^{2+\delta}$, and therefore they are above the critical regularity for axi-symmetric solutions in a stronger sense than the usual $C^{1,\beta}$ solutions considered in other results. The strategy requires a computer assisted proof in order to prove the invertibility of a certain linear operator.
     \end{itemize}

\subsection{Details about the proof}
\subsubsection{Strategy of the proof}
We use a similar strategy as the one considered in \cite{CM}. Namely, we consider solutions composed of an infinite number of compactly supported smooth vorticities
$$\sum_{n=0}^{\infty}\w_{n}(x,t)$$
which we call vortex layers, fulfilling that, if $n_{2}<n_{1}$, then $\w_{n_{1}}$ is much more concentrated around the origin than $\w_{n_{2}}$. This allows us to obtain a good approximation for the interaction between different layers by using, for example, Taylor series expansions.
This allows us to approximate the interactions between outer velocity and inner vorticity as
\begin{equation}\label{approx1}
\begin{aligned}
    \sum_{j=0}^{k-1} (\w_{j}* \mathcal K)\cdot\nabla \w_{k}&\approx U_{k}\cdot\nabla \w_{k},\\
    \w_{k}\cdot\nabla \sum_{j=0}^{k-1} \w_{j} * \mathcal K&\approx \w_{k}\cdot\nabla U_{k}
\end{aligned}
\end{equation}
where $U_{k}$ is a a useful approximation for $\sum_{j=0}^{k-1}u(\w_{j})$.
The interaction between inner velocity and outer vorticity is negligible,
because the inner vorticity exhibits a larger cancellation
that makes the velocity it generates very small far away from the inner layer:
\begin{equation}\label{approx2}
  \sum_{j=0}^{k-1}\w_{j}\cdot\nabla (\w_{k} * \mathcal K)\text{ and }(\w_{k} * \mathcal K)\cdot\nabla \sum_{j=0}^{k-1}\w_{j}\approx 0.  
\end{equation}

Furthermore, we will choose the $\w_{k}$ so that the self-interactions cancel to first order, so that
\begin{equation}\label{approx3}
    (\w_{k} * \mathcal K)\cdot\nabla\w_{k}\text{ and }\w_{k}\cdot\nabla (\w_{k} * \mathcal K)\approx 0.
\end{equation}
If we, for now, ignore the contribution of the fractional dissipation,
the approximations \eqref{approx1}, \eqref{approx2} and \eqref{approx3} transforms the equation \eqref{vorticity} to
\begin{equation}\label{wk}
    \p_{t}\w_{k}+U_{k}\cdot\nabla\w_{k}=\w\cdot\nabla U_{k}
\end{equation}
which can be also written as
\begin{equation}\label{wkF}
   \p_{t}\w_{k}+\sum_{j=0}^k (\w_{j} * \mathcal K)\cdot\nabla \w_{k}+(\w_{k} * \mathcal K)\cdot\nabla \sum_{j=0}^k \w_{j}
   =\w_{k}\cdot\nabla \sum_{j=0}^k (\w_{j} * \mathcal K)+\sum_{j=0}^k \w_{j}\cdot\nabla (\w_{k} * \mathcal K)+F_k 
\end{equation}
where the force is compensating the errors we created in our approximations. Note that, if we add all the equations \eqref{wkF}, we obtain a solution to forced Euler. Furthermore, if the $U_{k}$ are chosen appropriately, \eqref{wk} can be solved explicitly. We can then use this simplified evolution equations to construct the blow-up.

Note that, \textit{a priori}, the forces $F_{k}$ will have very bad regularity, and we can only hope that they belong to some well-posedness class if the approximations we have chosen for our evolution equation are very good. This is specially challenging considering the approximations need to remain valid along the whole evolution of the layers. Furthermore, since we are dealing with fractional Navier-Stokes, we also need to obtain useful approximations for $|\nabla|^{\alpha}\w_{k}$, specifically approximations that still allow us to solve our simplified model explicitly, that do not worsen the other approximations we are considering, and that do not change the qualitative behaviour of our solutions.
\subsubsection{Challenges in the proof}
As mentioned in the last subsection, this paper considers a similar scenario for the blow-up as in \cite{CM}, where we have an infinite number of vortex layers, each one with very different scales, and lower frequency vortex layers produce growth in the higher frequency ones. This construction, when considering carefully chosen vortex layers, allows for some useful cancellation in the interaction between layers, which gives us a first-order evolution equation.

However, the inclusion of diffusion makes the problem significantly more challenging. First, the diffusion is actively opposing the growth of the vorticity, and in order to overcome this, we need to consider solutions that blow-up in a stronger sense, namely, when considering fractional diffusion of order $\alpha$, the regularity of our velocity at the time of blow-up is less than $C^{1-\alpha}$ (in contrast, in \cite{CM}, we can take the regularity at the time of the blow-up arbitrarily close to $C^1$). This makes the behaviour of our solutions more singular, and therefore harder to control, and is specially concerning when dealing with the self-interactions of the vortex layers, where, in principle, we would get a self-interaction well below the well-posedness regime.

Furthermore, the diffusion also restrict our choices regarding the frequencies of the different layers. In particular, the construction employed in \cite{CM}, with exponential separation of the frequencies, fails to produce any meaningful growth in the presence of diffusion of any order. Not only that, but choosing frequencies that are consistent with growth in the presence of diffusion worsens the interactions between different layers, creating errors of higher order.

Finally, we need to obtain approximations for the fractional diffusion that we can include in our first order evolution without changing the structure of our solutions; otherwise this would interfere with the rest of the bounds for the evolution.
\subsubsection{Vortex layers}
Our blow-up solution consists of infinitely many vortex layers, each greater in magnitude, closer to the origin and existing for a shorter perioud than the previous one. Specifically, our vorticity $\omega$ is an infinite sum of $\omega_n$ which is nonzero only on the time interval $[T_n, 1]$. At the time of blow-up,
\begin{equation}\label{om-n}
\omega_n \approx A_n\psi(L_nx_1)\psi(L_nx_2)\psi(L_nx_3)\sin(M_nx_{n+2})\sin(M_nx_{n+3})e_{n+1}
\end{equation}
\footnote{The right-hand side is not divergence free, so we need some $e_{n+3}$ component to fix that,
which is small compared to the main term because of the oscillations in $x_{n+1}$ and $x_{n+3}$ have different frequencies.
}
where $A_n$, $M_n$ and $L_n \gg 1$, $\psi$ is a cut-off function and the subscripts are taken modulo 3.
Then the velocity field generated by $\omega_0, \dots, \omega_{n-1}$, especially the one generated by $\omega_{n-1}$:
\[
\omega_{n-1} \approx A_{n-1}\psi(L_{n-1}x_1)\psi(L_{n-1}x_2)\psi(L_{n-1}x_3)\sin(M_{n-1}x_{n+1})\sin(M_{n-1}x_{n+2})e_n,
\]
is roughly a constant times
\begin{align*}
U_n&\approx (A_{n-1}/M_{n-1})\psi(L_{n-1}x_1)\psi(L_{n-1}x_2)\psi(L_{n-1}x_3)\sin(M_{n-1}x_{n+1})\cos(M_{n-1}x_{n+2})e_{n+1}\\
&-(A_{n-1}/M_{n-1})\psi(L_{n-1}x_1)\psi(L_{n-1}x_2)\psi(L_{n-1}x_3)\cos(M_{n-1}x_{n+1})\sin(M_{n-1}x_{n+2})e_{n+2}.
\end{align*}
We will arrange that $L_n \gg L_{n-1}$, so $\omega_n$ is supported much closer to the origin than $\omega_{n-1}$.
Then the velocity felt by $\omega_n$ has a first order approximation near the origin:
\begin{equation}\label{Un-lin}
U_n \approx A_{n-1}x_{n+1}e_{n+1} - A_{n-1}x_{n+2}e_{n+2}.
\end{equation}
This is a saddle point, whose effect is stretching the $x_{n+1}$ direction while compressing the $x_{n+2}$ direction.
Thanks to the vorticity stretching term $\omega \cdot \nabla u$ in the vorticity formulation \eqref{vorticity},
$\omega_n$, whose main term has been made to align with $e_{n+1}$, is thus stretched and leads to the blow-up.

\subsubsection{Constrained choice of the parameters}
The choice of the parameters are constrained by a number of factors.
First, the stretching caused by the outer vortex layers to $\omega_n$
is of magnitude $A_{n-1}$. For the blow-up to happen, the stretching should beat the dissipation,
which is of magnitude $M_n^\alpha$, so we have our first constraint:
\[
M_n^\alpha \le A_{n-1}.
\]
Assume that $\omega_n$ is stretched by a factor of $K_n$ during the process.
Then when it first appears, it has the magnitude of $A_n/K_n$,
its frequency in the $x_{n+2}$ direction is $M_n/K_n$,
while its frequency in the $x_{n+3}$ direction is still $M_n$,
so its $C^s$ norm is then $A_nM_n^s/K_n$.
We will switch on $\omega_n$ purely by force,
which should remain bounded in the $C^s$ norm, so our second constraint is
\[
A_nM_n^s/K_n \le 1.
\]
In addition, for the approximation \eqref{Un-lin} to work,
the support of $\omega_n$ should not exceed the wavelength of $\omega_{n-1}$,
so that we can use the approximation $\sin x \sim x$.
This should hold throughout the period when $\omega_n$ exists,
especially at the beginning, when its support has size $K_n/L_n$ in one of the directions. Hence
\[
K_n/L_n \le 1/M_{n-1}.
\]
Finally, when $\omega_n$ evolves, the main term of error comes from its self-interaction,
which is corrected by force. Since $\omega_n$ has magnitude $A_n$,
the $C^s$ norm of the quadratic terms $u \cdot \nabla\omega$ and $\omega \cdot \nabla u$ are of sizes
$A_n^2M_n^s$. However, the sinusodial shape of the vorticity makes main terms cancel out (see Lemma \ref{quadratic} for details).
The error of the next order is smaller by a factor of $M_n/L_n$, i.e., of size $A_n^2L_nM_n^{s-1}$.
Because the strength of stretching per unit time is $A_{n-1}$, $\omega_n$ exists for a time period $1/A_{n-1}$,
so in order to bound the $L_t^1C_x^s$ norm of the error, we must have
\[
A_n^2L_nM_n^{s-1}/A_{n-1} \le 1.
\]

Solving the constraints one by one we get
\[
K_n \ge A_nM_n^s, \quad L_n \ge K_nM_{n-1} \ge A_nM_n^sM_{n-1}
\]
and finally
\[
A_n^3M_n^{2s-1}M_{n-1} \le A_{n-1}.
\]
It is now reasonable to take the equality in the constraint $M_n^\alpha \le A_{n-1}$. Then we get
\[
M_{n+1}^{3\alpha}M_n^{2s-1-\alpha}M_{n-1} \le 1
\]
or
\begin{equation}\label{recurrence}
3\alpha\ln M_{n+1} + (2s-1-\alpha)\ln M_n + \ln M_{n-1} \le 0.
\end{equation}
Since this is a linear recurrence inequality in terms of $\ln M_n$, it makes sense to take $\ln M_n = R^n$,
for which the characteristic equation must have a real root, so
\[
\Delta = (2s-1-\alpha)^2 - 12\alpha \ge 0.
\]
We must have $2s-1-\alpha<0$, otherwise the left-hand side of \eqref{recurrence} is positive (because $M_n \gg 1$), so
\[
s \le s(\alpha) := (1 + \alpha)/2 - \sqrt{3\alpha}
\]
so if there are $\alpha$ derivatives in the dissipation, the forcing upon $\omega_n$ can be controlled in the space $L_1^tC_x^{s(\alpha)-}$.
Now all we need is $s(\alpha) > 0$, which holds if $\alpha$ is a small positive number.
Note that this is only a back-of-the-envelope calculation:
in reality there are other sources of error getting in the way.

\subsubsection{Symmetry of the construction}
In our construction, the vorticity and the velocity it generates satisfies certain symmetry with respect to reflections across coordinate planes.
We say that a vector field $V(x)$ is \textbf{polar} if for all $j = 1, 2, 3$, $V_j(x)$ is odd in $x_j$,
and that $V(x)$ is \textbf{axial} if for all $j = 1, 2, 3$, $V_j(x)$ is even in $x_j$ but odd in $x_k$ for all $k \neq j$.
The vorticity of our blow-up will thus be axial and by the Biot--Savart law, the velocity it generates is polar.

\subsection{Outline of the paper}
In Section \ref{flowmap} we study how the flow generated by (a more delicate version of) \eqref{Un-lin} deforms the vortex layer $\omega_n$.
In Section \ref{om}, we estimate how this deformation introduces various errors in the linear and bilinear terms.
In Section \ref{construction}, we spell out the construction itself and show that all the errors are under control.

\section{Flow map estimates}\label{flowmap}
The velocity $U$ that the outer vortex layers generate is polar, meaning that $U_j$ is odd in $x_j$ ($j = 1, 2, 3$),
but we will not use $U$ itself, but the following approximation to generate the flow that transports the inner vortex layers:
\[
(x_1\partial_1U_1(0, x_2, 0, t), U_2(0, x_2, 0, t), x_3\partial_3U_3(0, x_2, 0, t)).
\]
In this section we study the flow generated by this field.
Note that the $e_2$ component only depends on $x_2$,
and the $e_j$ component ($j = 1, 3$) is linear in $x_j$;
thus the flow can be solved explicitly, first in the $e_2$ component, then in the other components.

\begin{lem}\label{1DODE}
Let $u(\cdot, t)$ be a smooth scalar function that is odd in the first variable.
Assume that $x \in \mathbb R$ and that for all $t \ge 0$, $x^2\sup|\partial_1^3u(\cdot, t)| \le -\partial_1u(0, t)$.
Then the solution to
\[
\phi(x, 0) = x, \quad \partial_t\phi(x, t) = u(\phi(x, t), t)
\]
exists for all $t \ge 0$ and satisfies:

(i) $|\phi(x, t)| \le |x|$.

(ii) $\phi(x, t)$ has the sign of $x$.
    
(iii) $|\phi(x, t)| \le |x|\exp\int_0^t (\partial_1u(0, s) + \frac{x^2\sup|\partial_1^3u(\cdot, s)|}6)ds$.
   
(iv) $|\phi(x, t)| \ge |x|\exp\int_0^t (\partial_1u(0, s) - \frac{x^2\sup|\partial_1^3u(\cdot, s)|}6)ds$.
    
(v) $\partial_x\phi(x, t) \le \exp\int_0^t (\partial_1u(0, s) + \frac{x^2\sup|\partial_1^3u(\cdot, s)|}2)ds \le 1$.
    
(vi) $\partial_x\phi(x, t) \ge \exp\int_0^t (\partial_1u(0, s) - \frac{x^2\sup|\partial_1^3u(\cdot, s)|}2)ds$.
    
(vii) $|\partial_{xx}\phi(x, t)| \le |x|\int_0^t \sup|\partial_1^3u(\cdot, s)|ds$.
\end{lem}
\begin{proof}
(i) For $x = 0$, the solution exists and equals 0 for all time. Now we suppose that $x \neq 0$.
Let $I$ be the maximal interval in $\mathbb R^+$ on which the solution exists and stays in $(-2|x|,2|x|)$.
Since $\phi(x, 0) = x$, by continuity, $I$ is non-empty.
For $t \in I$, by Taylor's theorem with the Lagrange remainder and the hypothesis on $\sup|\partial_1^3u(\cdot, t)|$,
\begin{align*}
|u(\phi(x, t), t) - \partial_1u(0, t)\phi(x, t)|
&\le \frac{|\phi(x, t)|^3}6\sup|\partial_1^3u(\cdot, t)|\\
&\le \frac{|\phi(x, t)|^3}6\frac{|\partial_1u(0, t)|}{x^2}
\le \frac{2|\partial_1u(0, t)\phi(x, t)|}{3}
\end{align*}
so $\partial_t\phi(x, t) = u(\phi(x, t), t)$ has the sign of $\partial_1u(0, t)\phi(x, t)$,
or that of $-\phi(x, t)$ (because $\partial_1u(0, t) \le 0$), so $|\phi(x, t)| \le |\phi(0, t)| \le |x|$.
If $I \neq \mathbb R$, then by local existence, the solution can be extended beyond $I$
while still staying in $(-2|x|,2|x|)$. Hence the solution exists for all $t \ge 0$,
and satisfies $|\phi(t)| \le |x|$.

(ii) Since $u$ is smooth, by the uniqueness theorem,
any solution with $x \neq 0$ does not reach 0,
so by continuity, $\phi(t)$ has the sign of $x$.

(iii) and (iv) By Taylor's theorem with the Lagrange remainder,
\begin{align*}
|u(\phi(x, t), t) - \partial_1u(0, t)\phi(x, t)|
&\le \frac{|\phi(x, t)|^3\sup|\partial_1^3u(\cdot, t)|}6\\
&\le \frac{x^2\sup|\partial_1^3u(\cdot, t)|}6|\phi(x, t)|.
\end{align*}
Multiplication by $2|\phi(x, t)|$ gives
\[
|\partial_t\phi(x, t)^2 - 2\partial_1u(0, t)\phi(x, t)^2|
\le \frac{x^2\sup|\partial_1^3u(\cdot, t)|}3\phi(x, t)^2
\]
so
\[
\partial_t\phi(x, t)^2 \le (2\partial_1u(0, t) + \frac{x^2\sup|\partial_1^3u(\cdot, t)|}3)\phi(x, t)^2
\]
so
\[
\phi(x, t)^2 \le x^2\exp\int_0^t (2\partial_1u(0, s) + \frac{x^2\sup|\partial_1^3u(\cdot, s)|}3)ds
\]
so
\[
|\phi(x, t)| \le |x|\exp\int_0^t (\partial_1u(0, s) + \frac{x^2\sup|\partial_1^3u(\cdot, s)|}6)ds.
\]
The inequality in the other direction follows similarly.

(v) and (vi) Taking the $x$ derivative of the ODE gives
\[
\partial_{tx}\phi(x, t) = \partial_1u(\phi(x, t), t)\partial_x\phi(x, t).
\]
Since $\partial_x\phi(x, 0) = 1$,
\[
\partial_x\phi(x, t) = \exp\int_0^t \partial_1u(\phi(x, s), s)ds.
\]
By Taylor's theorem with the Lagrange remainder,
\[
|\partial_1u(\phi(x, t), t) - \partial_1u(0, t)|
\le \frac{|\phi(x, t)|^2\sup|\partial_1^3u(\cdot, t)|}2
\le \frac{x^2\sup|\partial_1^3u(\cdot, t)|}2
\]
so
\[
\left| \int_0^t \partial_1u(\phi(x, s), s)ds - \int_0^t \partial_1u(0, s)ds \right|
\le \frac{x^2}{2}\int_0^t \sup|\partial_1^3u(\cdot, s)|ds
\]
which gives the first inequality of (v), as well as (vi). For the second inequality of (v),
Note that $x^2|\partial_1^3u(\cdot, t)| \le |\partial_1u(0, t)|$,
so $\partial_1u(0, s) + \frac{x^2\sup|\partial_1^3u(\cdot, s)|}2$
has sign of $\partial_1u(0, s) \le 0$.

(vii) Taking another $x$ derivative of the ODE gives
\[
\partial_{txx}\phi(x, t) = \partial_1^2u(\phi(x, t), t)\partial_x\phi(x, t)^2
+ \partial_1u(\phi(x, t), t)\partial_{xx}\phi(x, t).
\]
Since
\[
|\partial_1u(\phi(x, t), t) - \partial_1u(0, t)|
\le \frac{x^2\sup|\partial_1^3u(\cdot, t)|}2
\le \frac{\partial_1u(0, t)}2,
\]
$\partial_1u(\phi(x, t), t)$ has the sign of $\partial_1u(0, t)$, which in non-positive.
Then $\partial_1u(\phi(x, t), t)\partial_{xx}\phi(x, t)$ has the sign of $-\partial_{xx}\phi(x, t)$, so
\[
\partial_t|\partial_{xx}\phi(x, t)| \le \partial_1^2u(\phi(x, t), t)\partial_x\phi(x, t)^2.
\]
Using $|\partial_1^2u(\phi(x, t), t)| \le |x|\sup|\partial_1^3u(\cdot, t)|$ and (v) gives
\[
\partial_t|\partial_{xx}\phi(x, t)| \le |x|\sup|\partial_1^3u(\cdot, t)|.
\]
Since $\partial_{xx}\phi(x, 0) = 0$, (vii) follows.
\end{proof}

\begin{lem}\label{1DODE2}
Let $u(x, t)$ be a smooth scalar function that is odd in $x$.
Let $\phi(x, t_0, t)$ be its flow map, satisfying
\[
\phi(x, t_0, t_0) = x, \quad \partial_t\phi(x, t_0, t) = u(\phi(x, t_0, t)).
\]
Let
\[
K(t) = \exp\int_t^1 -\partial_xu(0, s)ds.
\]
Let $t \le 1$ and $X \ge 0$. Assume for all $s \in [t, 1]$, $X^2\sup|\partial_x^3u(\cdot, s)| \le -\partial_xu(0, s)$. Then

(i) For $|x| \le K(t)^{-1}X\exp\int_t^1 -\frac{X^2\sup|\partial_x^3u(\cdot, s)|}6ds$, $|\phi(x, 1, t)| \le X$.

(ii) For $|x| \le X$, $\partial_x\phi(x, t, 1) \le K(t)^{-1}\exp\int_t^1 \frac{x^2\sup|\partial_x^3u(\cdot, s)|}2ds \le 1$.

(iii) For $|x| \le X$, $|\partial_x^2\phi(x, t, 1)| \le |x|\int_t^1 \sup|\partial_x^3u(\cdot, s)|ds$.
\end{lem}
\begin{proof}
(i) Assume that $|\phi(x, 1, t)| > X$, or that the flow map cannot be extended to time $t$.
Since $\partial_xu(0, s) \le 0$, $K(t) \ge 1$, so $|\phi(x, 1, 1)| = |x| \le X$.
Then there is $t < t_1 \le 1$ such that $|\phi(x, 1, t_1)| = X$.
Since
\begin{align*}
\phi(\phi(x, 1, t_1), t_1, t_1) &= \phi(x, 1, t_1) = \pm X,\\
\partial_t\phi(\phi(x, 1, t_1), t_1, t) &= u(\phi(\phi(x, 1, t_1), t_1, t), t),
\end{align*}
$\phi(\phi(x, 1, t_1), t_1, t)$ satisfies the ODE in Lemma \ref{1DODE}, with initial data
$\pm X$ satisfying the assumption. Then by Lemma \ref{1DODE} (iv),
\begin{align*}
|\phi(\phi(x, 1, t_1), t_1, 1)| &\ge K(t_1)^{-1}X\exp\int_{t_1}^1 -\frac{X^2\sup|\partial_x^3u(\cdot, s)|}6ds\\
&\ge K(t)^{-1}X\exp\int_t^1 -\frac{X^2\sup|\partial_x^3u(\cdot, s)|}6ds.
\end{align*}
By the reversibility of the flow, $\phi(\phi(x, 1, t_1), t_1, 1) = x$.
Hence the inequality on $x$ in (i) is saturated.
If $X = 0$ then $x = 0$, so $\phi(x, 1, s) = 0$ by the uniqueness theorem.
Otherwise $K(t) = K(t_1)$, so for all $s \in [t, t_1]$, $\partial_xu(0, s) = 0$.
By the assumption, $\sup|\partial_x^3u(\cdot, s)| = 0$. Since $u$ is odd in $x$,
$u(\cdot, s) \equiv 0$ for all $s \in [t, t_1]$, so $|\phi(x, 1, t)| = |\phi(x, 1, t_1)| = X$, a contradiction.

(ii) and (iii) follow from Lemma \ref{1DODE} (v) and (vii) respectively.
\end{proof}

Now we solve the flow in the remaining components.
\begin{lem}\label{3DPDE}
Let $U(x, t)$ be a smooth, polar and divergence-free velocity field on $\mathbb R^3$ of the form
\[
U(x, t) = (x_1\partial_1U_1(0, x_2, 0, t), U_2(0, x_2, 0, t), x_3\partial_3U_3(0, x_2, 0, t)).
\]
Let $\phi(x, t_0, t)$ be its flow map, satisfying
\[
\phi(x, t_0, t_0) = x, \quad \partial_t\phi(x, t_0, t) = U(\phi(x, t_0, t)).
\]
Let
\[
K = \exp\int_t^1 -\partial_2U_2(0, s)ds.
\]
Let $t \le 1$. Assume for all $s \in [t, 1]$, $X^2\sup|\partial_2^3U_2(0, \cdot, 0, s)| \le -\partial_2U_2(0, 0, 0, s)$.
Then for $|x_2| \le X$,
\[
\phi_1(x, t, 1) = x_1Kg_1(x_2), \quad \phi_3(x, t, 1) = x_3g_3(x_2),
\]
where
\begin{align*}
|\ln g_3(x_2)| &\le \int_t^1 \sup|\partial_3U_3(0, \cdot, 0, s)|ds,\\
|\ln g_1(x_2)| &\le \int_t^1 (\sup|\partial_3U_3(0, \cdot, 0, s)|
+ \frac{x_2^2\sup|\partial_2^3U_2(0, \cdot, 0, s)|}2)ds,\\
|(\ln g_3)'(x_2)| &\le |x_2|\int_t^1 \sup|\partial_{223}U_3(0, \cdot, 0, s)|ds,\\
|(\ln g_1)'(x_2)|
&\le |x_2|\int_t^1 (\sup|\partial_{223}U_3(0, \cdot, 0, s)| + \sup|\partial_2^3U_2(0, \cdot, 0, s)|)ds,\\
|(\ln g_3)''(x_2)| &\le (1 + \ln K)\int_t^1 \sup|\partial_{223}U_3(0, \cdot, 0, s)|ds,\\
|(\ln g_1)''(x_2)|
&\le (1 + \ln K)\int_t^1 (\sup|\partial_{223}U_3(0, \cdot, 0, s)| + \sup|\partial_2^3U_2(0, \cdot, 0, s)|)ds.
\end{align*}
\end{lem}
\begin{proof}
We have
\begin{align*}
|\ln g_3(x_2)| &= |\int_t^1 \partial_3U_3(0, \phi_2(x, t, s), 0, s)ds|
\le \int_t^1 \sup|\partial_3U_3(0, \cdot, 0, s)|ds,\\
\ln g_1(x_2) &= \int_t^1 (\partial_1U_1(0, \phi_2(x, t, s), 0, s) + \partial_2U_2(0, s))ds\\
&= \int_t^1 (\partial_2U_2(0, s) - \partial_2U_2(0, \phi_2(x, t, s), 0, s))ds - \ln g_3(x_2).
\end{align*}
By Lemma \ref{1DODE} (i), $|\phi_2(x, t, s)| \le |x_2|$, so by Taylor's theorem with the Langrange remainder,
\[
|\partial_2U_2(0, s) - \partial_2U_2(0, \phi_2(x, t, s), 0, s)|
\le \frac{x_2^2\sup|\partial_2^3U_2(0, \cdot, 0, s)|}2
\]
so
\[
|\ln g_1(x_2)| \le \int_t^1 (\sup|\partial_3U_3(0, \cdot, 0, s)|
+ \frac{x_2^2\sup|\partial_2^3U_2(0, \cdot, 0, s)|}2)ds.
\]
Also,
\begin{align*}
|(\ln g_3)'(x_2)|
&= |\int_t^1 \partial_{23}U_3(0, \phi_2(x, t, s), 0, s)\partial_2\phi_2(x, t, s)ds|\\
&\le \int_t^1 |\partial_{23}U_3(0, \phi_2(x, t, s), 0, s)|ds \tag{by Lemma \ref{1DODE2} (ii)}\\
&\le |x_2|\int_t^1 \sup|\partial_{223}U_3(0, \cdot, 0, s)|ds, \tag{$\partial_{23}U_3(0, s) = 0$}\\
|(\ln g_1)'(x_2)|
&\le \int_t^1 |\partial_2^2U_2(0, \phi_2(x, t, s), 0, s)\partial_2\phi_2(x, t, s|)ds + |(\ln g_3)'(x_2)|\\
&\le |x_2|\int_t^1 (\sup|\partial_{223}U_3(0, \cdot, 0, s)| + \sup|\partial_2^3U_2(0, \cdot, 0, s)|)ds.
\end{align*}
Also,
\begin{align*}
|(\ln g_3)''(x_2)|
&\le \int_t^1 |\partial_{223}U_3(0, \phi_2(x, t, s), 0, s)|ds \tag{by Lemma \ref{1DODE2} (ii)}\\
&+ \int_t^1 |\partial_{23}U_3(0, \phi_2(x, t, s), 0, s)\partial_2^2\phi_2(x, t, s)|ds\\
&\le \int_t^1 \sup|\partial_{223}U_3(0, \cdot, 0, s)|ds\\
&+ X^2\int_t^1 \sup|\partial_{223}U_3(0, \cdot, 0, s)|ds\int_t^1 \sup|\partial_2^3U_2(0, \cdot, 0, s)|ds
\tag{by Lemma \ref{1DODE2} (iii)}\\
&\le \int_t^1 \sup|\partial_{223}U_3(0, \cdot, 0, s)|ds(1 - \int_t^1 \partial_2U_2(0, s)ds)
\tag{by assumption}\\
&= (1 + \ln K)\int_t^1 \sup|\partial_{223}U_3(0, \cdot, 0, s)|ds.
\end{align*}
Similarly,
\begin{align*}
|(\ln g_1)''(x_2)|
&\le \int_t^1 |\partial_2^3U_2(0, \phi_2(x, t, s), 0, s)|ds\\
&+ \int_t^1 |\partial_2^2U_2(0, \phi_2(x, t, s), 0, s)\partial_2^2\phi_2(x, t, s)|ds + |(\ln g_3)''(x_2)|\\
&\le (1 + \ln K)\int_t^1 (\sup|\partial_{223}U_3(0, \cdot, 0, s)| + \sup|\partial_2^3U_2(0, \cdot, 0, s)|)ds.
\end{align*}
\end{proof}

The behavior of the vortex layer transported by $U$ can now be studied.
We also include a term that corresponds to the dissipation.
\begin{lem}\label{3DPDE2}
Let $U$, $\phi$, $K$, $g_1$ and $g_3$ be as in Lemma \ref{3DPDE}.
For $i = 1, 3$, let $\omega$ be a smooth and axial vector field on $\mathbb R^3$.
Let $g$ be a function of $t$ only. Then the solution to the PDE
\[
\partial_t\omega_i + U \cdot \nabla \omega_i = \omega_i\partial_iU_i + g\omega_i
\]
is
\begin{align*}
\omega_1(x, t) &= (GK)^{-1}g_1(x_2)^{-1}\omega_1(\phi(x, t, 1), 1),\\
\omega_3(x, t) &= G^{-1}g_3(x_2)^{-1}\omega_3(\phi(x, t, 1), 1)
\end{align*}
where
\[
G = e^{\int_t^1 g(s)ds}.
\]
\end{lem}
\begin{proof}
Using $\partial_t(G\omega_i) = G(\partial_t\omega_i - g\omega_i)$
we can assume that $g = 0$ and $G = 1$.

Since $\omega_i$ is odd in $x_i$, it vanishes when $x_i = 0$. Since it is also smooth, we can put
$\omega_i = x_i\omega_i^\circ$, where $\omega_i^\circ$ is smooth. Then
\[
\partial_t\omega_i + U \cdot \nabla \omega_i
= x_i\partial_t\omega_i^\circ + x_iU \cdot \nabla \omega_i^\circ + \omega_i^\circ U \cdot \nabla x_i
= x_i\partial_t\omega_i^\circ + x_iU \cdot \nabla \omega_i^\circ + \omega_i\partial_iU_i
\]
where we have used
$\omega_i^\circ U \cdot \nabla x_i = \omega_i^\circ U_i = \omega_i^\circ x_i\partial_iU_i = \omega_i\partial_iU_i$.
After cancelling this term from both sides of the PDE we see that
$\omega_i^\circ = \omega_i/x_i$ is conserved along the flow line. Thus
\begin{align*}
\omega_1(x, t) &= x_1\omega_1(\phi(x, t, 1), 1)/\phi_1(x, t, 1) = K^{-1}g_1(x_2)^{-1}\omega_1(\phi(x, t, 1), 1),\\
\omega_3(x, t) &= x_3\omega_3(\phi(x, t, 1), 1)/\phi_3(x, t, 1) = g_3(x_2)^{-1}\omega_3(\phi(x, t, 1), 1).
\end{align*}
\end{proof}

\section{Velocity estimates}\label{om}
Let $\psi$ be a smooth and even bump function supported in $[-1, 1]$.
We first prescribe one component of the vorticity (with amplitude renormalized):
\begin{align*}
\omega_1 &= a(x_2)\psi(L_1g_1(x_2)x_1)\sin(Mg_2(x_2))\psi(L_2g_2(x_2))\sin(Mg_3(x_2)x_3)\\
&\times \psi(L_3g_3(x_2)x_3).
\end{align*}
(Actually $a = 1/g_1$, but this generality is useful later) and estimate the velocity (and its derivatives) it generates.
Then $\omega_1$ is supported in $\{|x| \le 1\}$.
Through out this section we assume that $M \ge 1$, $L_1, L_2, L_3 \in [1, M^{1-\delta}]$ ($0 < \delta \le 1$), $|a| \le 2$, $g_1, g_3 \in [1/2, 2]$ and $|g_2'| \le 1$, unless noted otherwise.

For any multi-index $J = (j_1, j_2, j_3)$ let $|J| = \sum_{i=1}^3 j_i$ and
$\partial_J = \partial_1^{j_i}\partial_2^{j_2}\partial_3^{j_3}$.

Let $\delta(j) = 1$ if $j = 0$ and 0 if $j > 0$.
We will approximate the velocity (and its derivatives of order $J$) generated by $\omega_1$ using $\tilde u_{,J} = (0, \tilde u_{2,J}, \tilde u_{3,J})$, where
\begin{align*}
\tilde u_{2,J} &=
\frac{\delta(j_1)g_2'(x_2)^{j_2}g_3(x_2)^{1+j_3}a(x_2)\psi(L_1g_1(x_2)x_1)\sin(Mg_2(x_2) + j_2\pi/2)}
{M^{1-|J|}(g_3(x_2)^2 + g_2'(x_2)^2)}\\
&\times \psi(L_2g_2(x_2))\cos(Mg_3(x_2)x_3 + j_3\pi/2)\psi(L_3g_3(x_2)x_3),\\
\tilde u_{3,J} &=
-\frac{\delta(j_1)g_2'(x_2)^{1+j_2}g_3(x_3)^{j_3}a(x_2)\psi(L_1g_1(x_2)x_1)\cos(Mg_2(x_2) + j_2\pi/2)}
{M^{1-|J|}(g_3(x_2)^2 + g_2'(x_2)^2)}\\
&\times \psi(L_2g_2(x_2))\sin(Mg_3(x_2)x_3 + j_3\pi/2)\psi(L_3g_3(x_2)x_3).
\end{align*}

Let $\mathcal K$ be the Biot--Savart kernel.

\begin{lem}\label{BS-C01}
For $|J| \le 1$,
\[
|\tilde u_{,J} - \partial_J(\omega_1e_1 * \mathcal K)| \lesssim_\delta BM^{(|J|-2)(1-\delta)}
\]
where $B = \sum_{i=1}^3 L_i + L_3\ln M + \sup|a'| + \sup|g_1'| + ML_3^{-1}(\sup|g_3'| + \sup|g_2''|)$.
\end{lem}
\begin{proof}
We can decompose all factors of the form $\sin X$ as a linear combination of $e^{\pm iX}$ and $\cos X$
as a linear combination of $e^{\pm i(X + \pi/2)} = \pm ie^{\pm iX}$ respectively,
so we can replace all factors of the form
$\sin X$ with $e^{iX}$ and $\cos X$ with $e^{i(X + \pi/2)} = ie^{iX}$ respectively,
in both $\omega_1$ and $\tilde u_{,J}$. We then have
\begin{align*}
\partial_J(\omega_1e_1 * \mathcal K)
&= \int_{\mathbb R^3} \frac{h\times e_1}{4\pi|h|^3}\partial_J\omega_1(x + h)dh\\
&= (-1)^{|J|}PV\int_{\mathbb R^3} \omega_1(x + h)\partial_J\frac{h\times e_1}{4\pi|h|^3}dh + c_J\omega_1(x)
\end{align*}
where $c_{(0,0,0)} = 0$, and if $J_j = 1$,
$c_J$ is a constant vector depending only on $J$ because the last term comes from the boundary term
$-\lim_{r\to0}\int_{S(r)} \frac{h\times e_1}{4\pi|h|^3}\omega_1(x+h)d\vec S\cdot e_j
=-\omega_1(x)\int_{S(1)} \frac{h\times e_1}{4\pi}h_jdS$.

Integrating by parts in $h_3$ we get, for any $k \in \mathbb N$, $n = 3, 5$ and $m = 0, 1, \dots, n - 2$,
\begin{align*}
&\left| \int_{\mathbb R} \frac{h_3^me^{iMg_3(x_2 + h_2)(x_3 + h_3)}\psi(L_3g_3(x_2 + h_2)(x_3 + h_3))}{|h|^n}dh_3 \right|\\
&= \left| \int_{\mathbb R} \frac{i^ke^{iMg_3(x_2 + h_2)(x_3 + h_3)}}{M^kg_3(x_2 + h_2)^k}
\frac{\partial^k}{\partial h_3^k}\frac{h^m\psi(L_3g_3(x_2 + h_2)(x_3 + h_3))}{|h|^n}dh_3 \right|\\
&\lesssim_k M^{-k}\sum_{j=0}^k L_3^j\int_{\mathbb R} |h|^{m-n-k+j}dh_3
\lesssim_k M^{-k}\sum_{j=0}^k L_3^j(h_1^2 + h_2^2)^{-(n-m-1+k-j)/2}
\end{align*}
Since each component of $\partial_J(h/|h|^3)$ is of the form $P(h)/|h|^{2|J|+3}$
where $P$ is a homogeneous polynomial of degree $|J| + 1$,
we can apply the above with $n = 2|J| + 3$ and $m = \deg_3P \le |J| + 1 \le n - 2$.
After including $|J| + 1 - m$ factors of $h_1$ and $h_2$ we get
\begin{align*}
&\left| \int_{h_1^2+h_2^2\ge M^{-2+2\delta}} \omega_1(x + h)\partial_J\frac{h\times e_1}{|h|^3}dh \right|\\
&\lesssim_k M^{-k}\sum_{j=0}^k L_3^j\int_{h_1^2+h_2^2\ge M^{-2+2\delta}\atop|x+h|\le1} (h_1^2 + h_2^2)^{-(|J|+1+k-j)/2}dh_1dh_2
\end{align*}
because $\omega$ is supported in $\{|x| \le 1\}$. Since this domain has area at most $\pi$, the right-hand side
\begin{align*}
&\lesssim_k M^{|J|-k}\sum_{j=0}^{k-2} L_3^jM^{(1-\delta)(k-j-1)} + M^{|J|-k}(L_3^{k-1}\ln M + L_3^k)\\
&\lesssim_k M^{|J|-\delta k}(1 + M^{-1+\delta}\ln M)
\lesssim_\delta M^{-2} \tag{$k := [\frac{|J| + 2}\delta] + 1$}
\end{align*}
so if we divide $\omega=\chi(M^{1-\delta}h)\omega+(1-\chi(M^{1-\delta}h))\omega$,
where $\chi(h) = \chi(h_1, h_2)$ is a smooth bump function equal to 1 near the origin, we can use the previous bounds to obtain good estimates for the contribution coming from $(1-\chi(M^{1-\delta}h))\omega$.

Denote $f(a) - f(b)$ by $f|_a^b$. Since $g_1 \le 2$ and $|J| \le 1$,
\begin{align*}
\left| \int \chi(M^{1-\delta}h)\omega_1|_{(x_1, x_2 + h_2, x_3 + h_3)}^{x + h}\partial_J\frac{h \times e_1}{4\pi|h|^3}dh \right|
&\lesssim \int \chi(M^{1-\delta}h)\frac{\left|\omega_1|_{(x_1, x_2 + h_2, x_3 + h_3)}^{x + h}\right|}{|h|^{2+|J|}}dh\\
&\lesssim L_1\int \frac{\chi(M^{1-\delta}h)|h_1|}{|h|^{2+|J|}}dh_3dh_1dh_2\\
&\lesssim L_1\int \frac{\chi(M^{1-\delta}h)|h_1|dh_1dh_2}{(h_1^2 + h_2^2)^{(1+|J|)/2}}\\
&\lesssim L_1M^{(|J|-2)(1-\delta)}
\end{align*}
so we can replace $\omega_1(x + h)$ by $\omega_1(x_1, x_2 + h_2, x_3 + h_3)$.

Note that when computing $\partial_2(a(x_2)\psi(L_1g_1x_1)\psi(L_2g_2)\psi(L_3g_3(x_3 + h_3))e^{iMg_3(x_2)(x_3 + h_3)})$,
the following factors appear:
\begin{center}
\begin{tabular}{|c|c|c|} 
 \hline
 Factor & Bound & $x_2$ derivative bound \\ 
 \hline
 $a(x_2)$ & 2 & $\sup|a'|$ \\ 
 \hline
 $\psi(L_1g_1(x_2)x_1)$ & $\lesssim 1$ & $\lesssim \sup|g_1'|$ ($|L_1g_1x_1| \lesssim 1$, so $L_1|x_1| \lesssim 1$) \\ 
 \hline
 $\psi(L_2g_2(x_2)$ & $\lesssim 1$ & $\lesssim L_2$ (note that $|g_2'| \le 1$) \\ 
 \hline
 $\psi(L_3g_3(x_2)(x_3 + h_3))$ & $\lesssim 1$ & $\lesssim \sup|g_3'|$ (similarly $L_3|x_3 + h_3| \lesssim 1)$ \\ 
 \hline
 $e^{iMg_3(x_2)(x_3 + h_3)}$ & 1 & $\lesssim ML_3^{-1}\sup|g_3'|$ ($M|x_3 + h_3| \lesssim ML_3^{-1})$ \\ 
 \hline
 The product & $\lesssim 1$ & $\lesssim B$ (we have used $M/L_3 \ge 1$) \\ 
 \hline
\end{tabular}
\end{center}
Also, on the support of $\chi(M^{1-\delta}h)$ we have
\begin{align*}
|\exp|_{iM(g_2(x_2) + h_2g_2'(x_2))}^{iMg_2(x_2 + h_2)}|
&\le Mh_2^2\sup|g_2''| < M^\delta|h_2|\sup|g_2''|\\
&\le ML_3^{-1}|h_2|\sup|g_2''| \le B|h_2|
\end{align*}
so
\begin{align*}
&|\omega_1(x_1, x_2 + h_2, x_3 + h_3) + a(x_2)\psi(L_1g_1(x_2)x_1)e^{iM(g_2(x_2) + h_2g_2'(x_2))}\psi(L_2g_2(x_2))\\
&\times e^{iMg_3(x_2)(x_3 + h_3)}\psi(L_3g_3(x_2)(x_3 + h_3))| \lesssim B|h_2|
\end{align*}
and
\begin{align*}
&\int \frac{\chi(M^{1-\delta}h)}{|h|^{2+|J|}}|\omega_1(x_1, x_2 + h_2, x_3 + h_3) - a(x_2)\psi(L_1g_1(x_2)x_1)e^{iM(g_2(x_2) + h_2g_2'(x_2))}\\
&\times \psi(L_2g_2(x_2))e^{iMg_3(x_2)(x_3 + h_3)}\psi(L_3g_3(x_2)(x_3 + h_3))|dh\\
&\lesssim B\int \frac{\chi(M^{1-\delta}h)|h_2|}{|h|^{2+|J|}}dh_3dh_1dh_2
\lesssim B\int \frac{\chi(M^{1-\delta}h)|h_2|dh_1dh_2}{(h_1^2+h_2^2)^{(1+|J|)/2}}\\
&\lesssim BM^{(|J|-2)(1-\delta)}.
\end{align*}
Since $g_3 \le 2$, $|\psi|_{L_3g_3(x_2)x_3}^{L_3g_3(x_2)(x_3 + h_3)}| \lesssim L_3|h_3|$. Then
\begin{align*}
&\int \frac{\chi(M^{1-\delta}h)}{|h|^{2+|J|}}|a(x_2)\psi(L_1g_1(x_2)x_1)e^{iM(g_2(x_2) + h_2g_2'(x_2))}\psi(L_2g_2(x_2))e^{iMg_3(x_2)(x_3 + h_3)}\\
&\times \psi|_{L_3g_3(x_2)x_3}^{L_3g_3(x_2)(x_3 + h_3)}|dh
\lesssim L_3\int\int_{|x_3+h_3|\le1} \frac{\chi(M^{1-\delta}h)|h_3|}{|h|^{2+|J|}}dh_3dh_1dh_2.
\end{align*}
because $\omega_1$ is supported in $\{|x| \le 1\}$.
Note that the integral in $h_3$ is restricted to an interval of length $O(1)$.
If $|h_3| \ge 1$ then the integrand is also $O(1)$, so the right-hand side
\begin{align*}
&\lesssim L_3M^{-2+2\delta} + L_3\int\int_0^1 \frac{\chi(M^{1-\delta}h)h_3dh_3}{|h|^{2+|J|}}dh_1dh_2\\
&\lesssim L_3M^{-2+2\delta} +
\begin{cases}
L_3\int \chi(M^{1-\delta}h)\ln\frac{h_1^2 + h_2^2 + 1}{h_1^2 + h_2^2}dh_1dh_2, & |J|=0,\\
L_3\int \chi(M^{1-\delta}h)(h_1^2 + h_2^2)^{-1/2}dh_1dh_2, & |J|=1
\end{cases}
\\
&\lesssim L_3M^{(|J|-2)(1-\delta)}(1 + \ln M)
\le BM^{(|J|-2)(1-\delta)}.
\end{align*}

Since the factors $a(x_2)\psi(L_1g_1(x_2)x_1)$, $e^{iMg_2(x_2)}$, $\psi(L_2g_2(x_2))$,
$e^{iMg_3(x_2)x_3}$ and $\psi(L_3g_3(x_2)x_3)$ are independent of $h$ and $O(1)$,
it suffices to show that
\begin{align*}
\int \chi(M^{1-\delta}h)e^{iMg_2'(x_2)h_2}e^{iMg_3(x_2)h_3}(-1)^{|J|}\partial_J\frac{h\times e_1}{4\pi|h|^3}dh
&+ c_J = \frac{i^{|J|+1}\delta(j_1)g_2'(x_2)^{j_2}g_3(x_2)^{j_3}}{M^{1-|J|}}\\
&\times\frac{(0, g_3(x_2), -g_2'(x_2))}{g_3(x_2)^2 + g_2'(x_2)^2} + \frac{O_\delta(1)}{M^{2-|J|}}.
\end{align*}
We first integrate by parts again, absorbing back the constant $c_J$.
If the derivative falls on $\chi$, integration by parts in $h_3$ shows that the integral in $h_3$ is
$\lesssim_k M^{-k}(h_1^2 + h_2^2)^{-(1+k)/2} \le M^{1-k\delta} \le M^{-2}$
(taking $k = [3/\delta] + 1$). Since the support of $\chi(M^{1-\delta}h)$ has size $O(M^{-2+2\delta})$
and its gradient has size $O(M^{1-\delta})$, the integral is under control.
Since taking the derivative of the exponential means multiplication by $ig_2'(x_2)$ or $ig_3(x_2)$,
it suffices to show the case when $J = 0$ and all factors involving $J$ disappear.
Then we undo the cross product with $e_1$,
changing the vector on the right-hand side to $(0, g_2'(x_2), g_3(x_2))$.
Next, using $h/|h|^3 = -\nabla(1/|h|)$, we can move one more derivative to the exponentials.
Then we change variables and it suffices to show that
\[
\left| \int_{\mathbb R^3} \frac{\chi(h/M^\delta)e^{ig_2'(x_2)h_2}e^{ig_3(x_2)h_3}}{4\pi|h|}dh
- \frac{1}{g_3(x_2)^2 + g_2'(x_2)^2} \right| \lesssim_\delta M^{-1}
\]
where $\chi(h) = \chi(h_1, h_2)$ is a smooth bump function equal to 1 near the origin.
Viewing $M^\delta$ as a whole, it suffices to show that, for $M \ge 1$ and $n \in \mathbb N$,
\[
\left| \int_{\mathbb R^3} \frac{\chi(h/M)e^{ig_2'(x_2)h_2}e^{ig_3(x_2)h_3}}{4\pi|h|}dh
- \frac{1}{g_3(x_2)^2 + g_2'(x_2)^2} \right| \lesssim_n M^{-n}.
\]

Let $\mathcal F_n$ be the Fourier transform in $\mathbb R^n$. Then it suffices to show that
\[
\left| \mathcal F_3\frac{\chi(h/M)}{4\pi|h|}(0, -g_2'(x_2), -g_3(x_2))
- \frac{1}{g_3(x_2)^2 + g_2'(x_2)^2} \right| \lesssim_n M^{-n}.
\]
Since the Fourier transform takes products to convolutions,
\[
\mathcal F_3\frac{\chi(h/M)}{4\pi|h|}
= \mathcal F_3(\chi(h/M)) * \mathcal F_3\frac{1}{4\pi|h|}
= M^3\mathcal F_3\chi(M\xi) * \frac{1}{|\xi|^2}.
\]
Since $\chi$ is independent of $h_3$, $\mathcal F_3\chi = \delta(\xi_3)\mathcal F_2\chi$.
Since $\int_{\mathbb R} \delta(M\xi_3)d\xi_3 = 1/M$,
\[
\mathcal F_3\frac{\chi(h/M)}{4\pi|h|}(\xi)
= M^2\int_{\mathbb R^2} \frac{\mathcal F_2\chi(M\eta)}{|\xi - \eta|^2}d\eta
\]
where $\xi - \eta = (\xi_1 - \eta_1, \xi_2 - \eta_2, \xi_3)$.

For $j + k > 0$, $\partial_1^j\partial_2^k\chi(0, 0) = 0$, so
$\int_{\mathbb R^2} \mathcal F_2\chi(M\eta)\eta_1^j\eta_2^kd\eta = 0$, so
\begin{align*}
&\left| \int_{\mathbb R^2} \frac{\mathcal F_2\chi(M\eta)}{|\xi - \eta|^2}d\eta
- \frac{1}{|\xi|^2}\int_{\mathbb R^2} \mathcal F_2\chi(M\eta)d\eta \right|\\
&\lesssim_n \int_{|\xi-\eta|\ge|\xi|/2} |\mathcal F_2\chi(M\eta)|\frac{|\eta|^nd\eta}{|\xi|^{n+2}}
\tag{when $|\eta| < |\xi|/2$ we use Taylor's theorem; otherwise it's trivial}\\
&+ \frac{1}{\xi_3^2}\int_{|\xi-\eta|<|\xi|/2} |\mathcal F_2\chi(M\eta)|d\eta
\tag{$|\xi|$, $|\xi - \eta| \ge |(\xi - \eta)_3| = |\xi_3|$}\\
&\lesssim_n \int_{\mathbb R^2} \frac{|\mathcal F_2\chi(\eta)||\eta|^n}{|M\xi|^{n+2}}d\eta_1d\eta_2
+ \frac{1}{\xi_3^2|M\xi|^{n+2}}\int_{\mathbb R^2} \frac{d\eta_1d\eta_2}{(1 + |\eta|)^3}
\lesssim_n \frac{1 + \xi_3^{-2}}{|M\xi|^{n+2}}.
\tag{$|\mathcal F_2\chi(M\eta)| \lesssim_n (1 + |M\eta|)^{-n-5}$, $M \ge 1$
and $|\eta| \ge |\xi| - |\xi - \eta| > |\xi|/2$}
\end{align*}
Since $\int_{\mathbb R^2} \mathcal F_2\chi(M\eta)d\eta = \chi(0)M^{-2} = M^{-2}$,
\[
\left| \mathcal F_3\frac{\chi(h/M)}{4\pi|h|}(\xi) - \frac{1}{|\xi|^2} \right|
\lesssim_n \frac{1 + \xi_3^{-2}}{M^n|\xi|^{n+2}}
\]
giving the desired bound since $|\xi| \ge |\xi_3| = g_3(x_2) \ge 1/2$.
\end{proof}
\begin{cor}\label{BS-C01-cor}
For $j = 0, 1$,
\begin{align*}
\|\tilde u\|_{C^j} &\lesssim (M + B)^j/M,\\
\|\tilde u - \omega_1e_1 * \mathcal K\|_{C^j} &\lesssim_\delta BM^{(j-2)(1-\delta)}
\end{align*}
because $|\tilde u_{,J}| \lesssim M^{|J|-1}$ and for $|J| = 1$,
$|\partial_J\tilde u - \tilde u_{,J}| \lesssim B/M$.
\end{cor}

Now we use the actual $\omega_1$, with $a = 1/g_1$, so $\sup|a'| \lesssim \sup|g_1'|$,
and the respective bound has $B$ replaced by $B_0 := B$ without the term $\sup|a'|$, i.e.,
$B_0 = \sum_{i=1}^3 L_i + L_3\ln M + \sup|g_1'| + ML_3^{-1}(\sup|g_3'| + \sup|g_2''|)$.
\begin{lem}\label{BS-C12}
For $|J| \le 1$,
\[
|\nabla \tilde u_{,J} - \partial_J\nabla(\omega_1e_1 * \mathcal K)| \lesssim_\delta B_1M^{(|J|-2)(1-\delta)}
\]
where $B_1 = B_0(M + B_0) + \sup|g_1''| + ML_3^{-1}\sup|g_3''|$.
\end{lem}
\begin{cor}\label{BS-C12-cor}
Similarly $\|\tilde u_{,J}\|_{C^1} \lesssim M + B_0$ and
$\|\omega_1e_1 * \mathcal K\|_{C^2} \lesssim_\delta M + B_1M^{-1+\delta}$.
\end{cor}
\begin{proof}[Proof of Lemma \ref{BS-C12}]
When we take derivatives of
\[
g_1(x_2)^{-1}\psi(L_1g_1(x_2)x_1)\sin(Mg_2(x_2))\psi(L_2g_2(x_2))\sin(Mg_3(x_2)x_3)\psi(L_3g_3(x_2)x_3)
\]
in various directions, they may fall on different factors.
In addition to changing $\psi$ to $\psi'$ and sin to cos, which do not affect the form of the bound,
it will also change $g_1(x_2)^{-1}$ to different factors $a(x_2)$, leading to different bounds, as detailed below:
\begin{center}
\begin{tabular}{|c|c|c|c|} 
 \hline
 Factor hit by $\partial$ & $a(x_2)$ & $|a'(x_2)| \lesssim$ & Effect on $B_0 \lesssim_\delta$ \\ 
 \hline
 $\partial_1\psi(L_1g_1x_1)$ & $L_1$ & 0 & $B_0L_1 \le B_0^2$ \\ 
 \hline
 $\partial_3\sin(Mg_3x_3)$ & $Mg_3/g_1$ & $M(\sup|g_1'| + \sup|g_3'|)$ & $B_0M$ \\
 \hline
 $\partial_3\psi(L_3g_3x_3)$ & $L_3g_3/g_1$ & $L_3(\sup|g_1'| + \sup|g_3'|)$ & $B_0L_3 \le B_0^2$ \\
 \hline
 $\partial_2(1/g_1(x_2))$ & $-g_1'/g_1^2$ & $\sup|g_1''| + \sup g_1'^2$ & $B_0\sup|g_1'| + \sup|g_1''|$ \\
 \hline
 $\partial_2\sin(Mg_2(x_2))$ & $Mg_2'/g_1$ & $M(\sup|g_1'| + \sup|g_2''|)$ & $B_0M$ \\
 \hline
 $\partial_2\psi(L_2g_2(x_2))$ & $L_2g_2'/g_1$ & $L_2(\sup|g_1'| + \sup|g_2''|)$ & $B_0L_2 \le B_0^2$ \\
 \hline
\end{tabular}
\end{center}
Since
\[
\partial_2\psi(L_1g_1(x_2)x_1) = L_1g_1(x_2)x_1\psi'(L_1g_1(x_2)x_1)(g_1'/g_1)(x_2),
\]
$|L_1x_1| \lesssim 1$ on the support of $\psi$
and $x\psi'(x)$ is smooth and compactly supported, we get an effect similar to that of $\partial_2(1/g_1(x_2))$.
Similarly the effect of $\partial_2\psi(L_3g_3(x_2)x_3)$ on $B_0$ is $B_0\sup|g_3'| + \sup|g_3''|$.
Also due to the support of $\psi(L_3g_3(x_2)x_3)$, the effect of $\partial_2\sin(Mg_3(x_2)x_3)$ is $O(M/L_3)$ times larger,
so it is $\lesssim_\delta B_0ML_3^{-1}\sup|g_3'| + ML_3^{-1}\sup|g_3''| \le B_0^2 + ML_3^{-1}\sup|g_3''| \le B_1$.

Finally in $\partial_2\tilde u_{,J}$ with $J \neq (1, 0, 0)$, there are extra terms coming from
\[
\left| \partial_2\frac{(g_2'(x_2)^{j_2}g_3(x_2)^{1+j_3}, g_2'(x_2)^{1+j_2}g_3(x_2)^{j_3})}{g_3(x_2)^2 + g_2'(x_2)^2} \right|
\lesssim \sup|g_3'| + \sup|g_2''| \le B_0
\]
contributing $B_0M^{|J|-1} \le B_1M^{|J|-2} \le B_1M^{(|J|-2)(1-\delta)}$ to the bound.
\end{proof}

Now we add another component to the vorticity. Let $\omega = (\omega_1, 0, \omega_3)$, where
\begin{align*}
\omega_3 &= \int_{-\infty}^{x_3} \partial_1\omega_1(x_1, x_2, s)ds
\end{align*}
so that $\omega$ is divergence-free. We first control the error in the quadratic self-interaction.
\begin{lem}\label{quadratic}
For $s \in [0, 1]$,
\[
\|(\omega * \mathcal K) \cdot \nabla\omega\|_{C^s} + \|\omega \cdot \nabla(\omega * \mathcal K)\|_{C^s}
\lesssim_\delta B_0^{1-s}B_1^sM^{-2+2\delta}(M + B_0).
\]
\end{lem}
\begin{proof}
Recall
\begin{align*}
\omega_1 &= -g_1(x_2)^{-1}\psi(L_1g_1(x_2)x_1)\sin(Mg_2(x_2))\psi(L_2g_2(x_2))\\
&\times \sin(Mg_3(x_2)x_3)\psi(L_3g_3(x_2)x_3).
\end{align*}
Each factor is $O(1)$, and their higher derivative bounds are (constants ignored):
\begin{center}
\begin{tabular}{|c|c|c|c|} 
 \hline
 Term & First derivative & Second derivative \\ 
 \hline
 $1/g_1(x_2)$ & $\sup|g_1'|$ & $\sup g_1'^2 + \sup|g_1''|$ \\
 \hline
 $\psi(L_1g_1x_1)$ & $L_1 + \sup|g_1'|$ & $L_1^2 + \sup g_1'^2 + \sup|g_1''|$ \\ 
 \hline
 $\sin(Mg_2)$ & $M$ & $M^2 + M\sup|g_2''|$ \\
 \hline
 $\psi(L_2g_2)$ & $L_2$ & $L_2^2 + L_2\sup|g_2''|$ \\ 
 \hline
 $\psi(L_3g_3x_3)$ & $L_3 + \sup|g_3'|$ & $L_3^2 + \sup g_3'^2 + \sup|g_3''|$ \\ 
 \hline
 $\cos(Mg_3x_3)$ & $M(1 + \sup|g_3'|/L_3)$ & $M^2(1 + \sup g_3'^2/L_3^2) + ML_3^{-1}\sup|g_3''|$ \\
 \hline
 $\omega_1$ & $M + B_0$ & $M^2 + B_1$ \\
 \hline
\end{tabular}
\end{center}
Similarly, the corresponding bounds for $\tilde u$ are smaller by a factor of $M$.

Note that $\partial_2\omega_1 - Mg_2'(x_2)\omega_1\cot(Mg_2(x_2))$ is a sum of functions like $\omega_1$ with the following factor replacements,
with corresponding derivative bounds:
\begin{center}
\begin{tabular}{|c|c|c|c|} 
 \hline
 Factor & Replaced by & sup & First derivative \\ 
 \hline
 $1/g_1$ & $-g_1'/g_1^2$ & $\sup|g_1'|$ & $\sup g_1'^2 + \sup|g_1''|$ \\ 
 \hline
 $\psi(L_1g_1x_1)$ & $L_1g_1'\psi'(L_1g_1x_1)x_1$ & $\sup|g_1'|$ & $L_1\sup|g_1'| + \sup g_1'^2 + \sup|g_1''|$ \\ 
 \hline
 $\psi(L_2g_2)$ & $L_2g_2'\psi'(L_2g_2)$ & $L_2$ & $L_2^2 + L_2\sup|g_2''|$ \\ 
 \hline
 $\psi(L_3g_3x_3)$ & $L_3g_3'\psi'(L_3g_3x_3)x_3$ & $\sup|g_3'|$ & $L_3\sup|g_3'| + \sup g_3'^2 + \sup|g_3''|$ \\
 \hline
 $\sin(Mg_3x_3)$ & $Mg_3'\cos(Mg_3x_3)x_3$ & $ML_3^{-1}\sup|g_3'|$ & $* \le B_1$ \\
 \hline
\end{tabular}
\end{center}
where $* = M^2L_3^{-1}\sup|g_3'| + M^2L_3^{-2}\sup g_3'^2 + ML_3^{-1}\sup|g_3''|$.
Then for $j = 0, 1$, using $\|\tilde u\|_{C^j} \lesssim (M + B)^j/M$ and $B_0(M + B_0) \le B_1$ we have that
\begin{align*}
\|\partial_2\omega_1 - Mg_2'(x_2)\omega_1\cot(Mg_2(x_2))\|_{C^j} &\lesssim B_j,\\
\|\tilde u_2\partial_2\omega_1 - \tilde u_2Mg_2'(x_2)\omega_1\cot(Mg_2(x_2))\|_{C^j} &\lesssim B_j/M.
\end{align*}
Also, $\partial_3\omega_1 - Mg_3(x_2)\omega_1\cot(Mg_3(x_2)x_3)$ is $\omega_1$ with the following factor replaced,
with corresponding bounds on $C^0$ and $C^1$ norms:
\begin{center}
\begin{tabular}{|c|c|c|c|} 
 \hline
 Factor & Replaced by & sup & First derivative \\ 
 \hline
 $\psi(L_3g_3x_3)$ & $L_3g_3\psi'(L_3g_3x_3)$ & $L_3$ & $L_3^2 + L_3\sup|g_3'|$ \\
 \hline
\end{tabular}
\end{center}
Then for $j = 0, 1$, using $L_3 + \sup|g_3'| \le B_0$ and $\|\tilde u\|_{C^j} \lesssim (M + B_0)^j/M$ we have that
\begin{align*}
\|\partial_3\omega_1 - Mg_3(x_2)\omega_1\cot(Mg_3(x_2)x_3)\|_{C^j} &\lesssim L_3(M + B_0)^j,\\
\|\tilde u_3\partial_3\omega_1 - \tilde u_3Mg_3(x_2)\omega_1\cot(Mg_3(x_2)x_3)\|_{C^j} &\lesssim L_3(M + B_0)^j/M.
\end{align*}
Since $L_3 \le B_0$ and $M$, $L_3(M + B_0) \le 2MB_0 \le 2B_1$, so the first bound dominates the second one.
Since $\tilde u_2g_2'(x_2)\cot(Mg_2(x_2)) + \tilde u_3g_3(x_2)\cot(Mg_3(x_2)x_3) = 0$,
\[
\|\tilde u\cdot\nabla\omega_1\|_{C^j} \lesssim B_j/M.
\]

Integrating by parts we have
\begin{align*}
\omega_3 &= -L_1\psi'(L_1g_1(x_2)x_1)\sin(Mg_2(x_2))\psi(L_2g_2(x_2))I,\\
I &= \int_{-\infty}^{x_3} \sin(Mg_3(x_2)y)\psi(L_3g_3(x_2)y)dy = I_0 + R_0,\\
I_0 &= -\frac{\cos(Mg_3(x_2)x_3)\psi(L_3g_3(x_2)x_3)}{Mg_3(x_2)},\\
R_0 &= \int_{-\infty}^{x_3} \frac{\cos(Mg_3(x_2)y)L_3\psi'(L_3g_3(x_2)y)}{M}dy\\
&= \int_{-\infty}^{L_3x_3} \frac{\cos(ML_3^{-1}g_3(x_2)y)\psi'(g_3(x_2)y)}{M}dy.
\end{align*}
The derivative bounds are:
\begin{center}
\begin{tabular}{|c|c|c|c|} 
 \hline
 Term & sup & First derivative & Second derivative \\ 
 \hline
 $\psi(L_3g_3x_3)$ & 1 & $L_3 + \sup|g_3'|$ & $L_3^2 + \sup g_3'^2 + \sup|g_3''|$ \\ 
 \hline
 $\cos(Mg_3x_3)$ & 1 & $M(1 + \sup|g_3'|/L_3)$ & $M^2(1 + \sup g_3'^2/L_3^2) + ML_3^{-1}\sup|g_3''|$ \\
 \hline
 $1/g_3(x_2)$ & 1 & $\sup|g_3'|$ & $\sup g_3'^2 + \sup|g_3''|$ \\
 \hline
 $I_0$ & $1/M$ & $1 + \sup|g_3'|/L_3$ & $M(1 + \sup g_3'^2/L_3^2) + \sup|g_3''|/L_3$ \\
 \hline
 $R_0$ & $1/M$ & $\frac{1 + \sup|g_3'|}{L_3} + L_3/M$ & $\frac{M^2(1 + \sup g_3'^2)}{L_3^2} + \sup|g_3''|/L_3 + L_3 + \sup|g_3'|$ \\
 \hline
 $I$ & $1/M$ & $1 + \sup|g_3'|/L_3$ & $M(1 + \sup g_3'^2/L_3^2) + \sup|g_3''|/L_3$ \\
 \hline
\end{tabular}
\end{center}
where we have used $L_3 + \sup|g_3'| \le M(1 + \sup|g_3'|/L_3) \lesssim M(1 + \sup g_3'^2/L_3^2)$
to dominate $R_0$ with $I_0$. Continuing,
\begin{center}
\begin{tabular}{|c|c|c|c|} 
 \hline
 Term & sup & First derivative & Second derivative \\ 
 \hline
 $\psi'(L_1g_1x_1)$ & 1 & $L_1 + \sup|g_1'|$ & $L_1^2 + \sup g_1'^2 + \sup|g_1''|$ \\ 
 \hline
 $\sin(Mg_2)$ & 1 & $M$ & $M^2 + M\sup|g_2''|$ \\
 \hline
 $\psi(L_2g_2)$ & 1 & $L_2$ & $L_2^2 + L_2\sup|g_2''|$ \\ 
 \hline
 $I$ & $1/M$ & $1 + \sup|g_3'|/L_3$ & $M(1 + \sup g_3'^2/L_3^2) + \sup|g_3''|/L_3$ \\
 \hline
 $M\omega_3/L_1$ & 1 & $M + B_0$ & $M^2 + B_1$ \\
 \hline
\end{tabular}
\end{center}
Then for $j = 0, 1$, using $\|\tilde u\|_{C^j} \lesssim (M + B)^j/M$ and $(M + B_0)^2 \le M^2 + B_1$ we have that
\begin{align*}
\|\nabla\omega_3\|_{C^j} &\lesssim L_1(M^{j+1} + B_j)/M,\\
\|\tilde u\cdot\nabla\omega_3\|_{C^j} &\lesssim L_1(M^{j+1} + B_j)/M^2.
\end{align*}
Since $L_1M^j \le B_j$ and $L_1 \le M$, $L_1(M^{j+1} + B_j) \le B_j(L_1 + M) \le 2B_jM$, so
\[
\|\tilde u\cdot\nabla\omega_3\|_{C^j} \lesssim B_j/M,\quad
\|\tilde u\cdot\nabla\omega\|_{C^j} \lesssim B_j/M.
\]
Since $\|\tilde u - \omega_1e_1 * \mathcal K\|_{C^j} \lesssim_\delta B_0M^{(j-2)(1-\delta)}$,
$\|\nabla\omega\|_{C^j} \lesssim M^{j+1} + B_j$ and $M(M + B_0) \le M^2 + B_1$,
\[
\|(\tilde u - \omega_1e_1 * \mathcal K)\cdot\nabla\omega\|_{C^j} \lesssim_\delta B_0M^{-2+2\delta}(M^{j+1} + B_j)
\]
so using $B_0M \le B_1$, $\|(\omega_1e_1 * \mathcal K)\cdot\nabla\omega\|_{C^j} \le B_jM^{-2+2\delta}(M + B_0)$.

Turning to
\begin{align*}
\omega_3 &= -L_1\psi'(L_1g_1(x_2)x_1)\sin(Mg_2(x_2))\psi(L_2g_2(x_2))I,\\
I &= \int_{-\infty}^{x_3} \sin(Mg_3(x_2)s)\psi(L_3g_3(x_2)y)dy = \sum_{i=0}^k I_i + R_k,\\
I_i &= -\frac{\cos(Mg_3(x_2)x_3 + i\pi/2)L_3^i\psi^{(i)}(L_3g_3(x_2)x_3)}{M^{i+1}g_3(x_2)},\\
R_k &= \int_{-\infty}^{L_3x_3} \frac{\cos(ML_3^{-1}g_3(x_2)y + k\pi/2)L_3^k\psi^{(k+1)}(g_3(x_2)y)}{M^{k+1}}dy.
\end{align*}
The term involving $I_i$ has the same form as $\omega_1$, with
\[
a(x_2) = L_1L_3^i/(M^{i+1}g_3(x_2))
\]
so its convolution with the Biot--Savart kernel has an estimate whose main term and error term are both
$L_1L_3^i/M^{i+1}$ times that of $\omega_1$. Then for $j = 0, 1$,
\[
\|(\text{part of $\omega_3$ involving $I_i$})e_3 * \mathcal K\|_{C^j}
\lesssim_\delta L_1L_3^i(M + B_0)M^{(j-2)(1-\delta)-i-1}.
\]
Let $k = [3/\delta] + 1$. since $L_i \le M$ ($i = 1, 2, 3$),
\[
\|\text{part of $\omega_3$ involving $R_k$}\|_{C^j}
\lesssim L_1L_3^kM^{j-k-1} \lesssim_\delta L_1/M^3
\]
so, since supp $\omega_3 \subset \{|x| \le 1\}$ and the Biot--Savart kernel is locally integrable,
\[
\|(\text{part of $\omega_3$ involving $R_k$})e_3 * \mathcal K\|_{C^j}
\lesssim_\delta L_1M^{-3}.
\]
Since $B_j \ge 1$,
summing over $I_0, \dots, I_k$ and $R_k$ we get, for $j = 0, 1$,
\[
\|\omega_3e_3 * \mathcal K\|_{C^j} \lesssim_\delta L_1(M + B_0)M^{(j-2)(1-\delta)-1}.
\]
Since $\|\nabla\omega\|_{C^j} \lesssim M^{j+1} + B_j$, $M(M + B_0) \le M^2 + B_1$ and $L_1(M^{j+1} + B_j) \le B_j(L_1 + M) \le 2B_jM$,
\[
\|(\omega_3e_3 * \mathcal K)\cdot\nabla\omega\|_{C^j} \lesssim_\delta L_1(M^{j+1} + B_j)M^{-3+2\delta}(M + B_0)
\le B_jM^{-2+2\delta}(M + B_0)
\]
so $(\omega * \mathcal K)\cdot\nabla\omega$ also has the same bound.

Turning to the other term, we have
\[
\|\partial_1\tilde u\|_{C^j} \lesssim L_1M^{j-1}. \tag{$j = 0, 1$}
\]
Since $\partial_1\omega_1$ is of the same form as $\omega_1$, with the factor $1/g_1$ replaced by $L_1$,
\[
\|\partial_1(\tilde u - \omega_1e_1 * \mathcal K)\|_{C^j} \lesssim_\delta L_1B_0M^{(j-2)(1-\delta)} \tag{$j = 0, 1$}
\]
so
\[
\|\partial_1(\omega_1e_1 * \mathcal K)\|_{C^j} \lesssim_\delta L_1(M + B_0)M^{(j-2)(1-\delta)}. \tag{$j = 0, 1$}
\]
Since $\|\omega_1\|_{C^j} \lesssim (M + B_0)^j$ ($j = 0, 1$),
\[
\|\omega_1\partial_1(\omega_1e_1 * \mathcal K)\|_{C^j} \lesssim_\delta L_1(M + B_0)^{j+1}M^{-2+2\delta}. \tag{$j = 0, 1$}
\]
Since $\partial_3\omega_1$ is of the same form as $\omega_1$, with the factor $1/g_1$ replaced by $Mg_3/g_1$ and $L_3g_3/g_1$,
the bounds for $\partial_3(\omega_1e_1 * \mathcal K)$ are $M/L_1$ times those of $\partial_1(\omega_1e_1 * \mathcal K)$.
Since the bounds for $\omega_3$ are $L_1/M$ times those of $\omega_1$,
$\omega_3\partial_3(\omega_1e_1 * \mathcal K)$ has the same bounds as $\omega_1\partial_1(\omega_1e_1 * \mathcal K)$,
so $\omega\cdot\nabla(\omega_1e_1 * \mathcal K)$ also has the same bounds.

Since $\partial_1\omega_1$ is of the same form as $\omega_1$, with the factor $1/g_1$ replaced by $L_1$, for $j = 0, 1$,
\[
\|\partial_1(\omega_3e_3 * \mathcal K)\|_{C^j} \lesssim_\delta L_1^2(M + B_0)M^{(j-2)(1-\delta)-1}.
\]
Since $\|\omega_1\|_{C^j} \lesssim (M + B_0)^j$,
\[
\|\omega_1\partial_1(\omega_3e_3 * \mathcal K)\|_{C^j} \lesssim_\delta L_1^2(M + B_0)^{j+1}M^{-3+2\delta}.
\]
Similarly $\omega_3\partial_3\nabla(\omega_3e_3 * \mathcal K)$ also has the same bounds.
Since $L_1(M + B_0)^j \le M(M + B_0)^j \le B_1$,
\[
\|\omega \cdot \nabla(\omega * \mathcal K)\|_{C^j} \lesssim_\delta L_1(M + B_0)^{j+1}M^{-2+2\delta}
\le B_j(M + B_0)M^{-2+2\delta}. \tag{$j = 0, 1$}
\]
The results now follow from interpolation and $B_0 \le MB_0 \le B_1$.
\end{proof}

\begin{lem}\label{BS-C1}
Assume that $g_2(0) = 0$. Then
\begin{align*}
\left| \partial_3(\omega * \mathcal K)_3(0, 0, 0) + \frac{g_2'(0)g_3(0)}{g_1(0)(g_3(0)^2 + g_2'(0)^2)} \right|
&\lesssim_\delta B_0M^{\delta-1},\\
|\nabla(\omega * \mathcal K)_1| &\lesssim_\delta L_1(M + B_0)M^{\delta-2}.
\end{align*}
\end{lem}
\begin{proof}
In the proof of Lemma \ref{quadratic} we saw
$\|\omega_3e_3 * \mathcal K\|_{C^1} \lesssim_\delta L_1(M + B_0)M^{\delta-2}$.
Since $(\omega_1e_1 * \mathcal K)_1 = 0$, we get the second bound.
The second term in the first bound is exactly $-\partial_3\tilde u_3(0, 0, 0)$,
so by combining Lemma \ref{BS-C01} and the $C^1$ bound for $\omega_3e_3 * \mathcal K$,
we get the first bound, which dominates the second one because $L_1M$ and $L_1B_0 \le B_0M$.
\end{proof}

Then we control the error in the bilinear interaction between the outer and inner vortex layers.
Let $D = \max\{|x|: x \in \supp\omega\}$.
\begin{lem}\label{farfield}
There is an absolute constant $C$ such that if $|x| > CD$,
then for all $j, n = 0, 1, \dots$, $|\nabla^j(\omega * \mathcal K)(x)| \lesssim_{j,n,\delta} M^{-n}$.
\end{lem}
\begin{proof}
We first show the result for $j = 0$. Recall that
\[
\omega_1e_1 * \mathcal K = \int_{\mathbb R^3} \frac{h\times e_1}{4\pi|h|^3}\omega_1(x + h)dh.
\]
In the proof of Lemma \ref{BS-C01} we have shown, for any $k \in \mathbb N$,
\begin{align*}
&\left| \int_{\mathbb R} \frac{h\times e_1\sin(Mg_3(x_2 + h_2)(x_3 + h_3))\psi(L_3g_3(x_2 + h_2)(x_3 + h_3))}{|h|^3}dh_3 \right|\\
&\lesssim_k M^{-k}\sum_{j=0}^k L_3^j(h_1^2 + h_2^2)^{-(1+k-j)/2}.
\end{align*}
Since $D \gtrsim 1/L_1 \ge M^{-1+\delta}$, if $|x_i|$ or $|x_2| > CD$ then
on the support of $\omega_1$, $h_1^2 + h_2^2 \ge M^{-2+2\delta}$,
so we get the desired bound as in the proof of Lemma \ref{BS-C01}.
If $|x_3| > CD$, then similarly on the support of $\omega_1$, $|h_3| \ge M^{-1+\delta}$,
so when integrating by parts in $h_3$ we get
\begin{align*}
&\left| \int_{\mathbb R} \frac{h\times e_1\sin(Mg_3(x_2 + h_2)(x_3 + h_3))\psi(L_3g_3(x_2 + h_2)(x_3 + h_3))}{|h|^3}dh_3 \right|\\
&\lesssim_k M^{-k}\sum_{j=0}^k L_3^j\int_{M^{-1+\delta}}^\infty \frac{|h_1| + |h_2| + h_3}{|h|^{3+k-j}}dh_3
\le M^{-k}\sum_{j=0}^k L_3^j\int_{M^{-1+\delta}}^\infty \frac{3dh_3}{|h|^{2+k-j}}\\
&\lesssim_k M^{1-k}\sum_{j=0}^k L_3^jM^{(1-\delta)(k-j)} \lesssim_k M^{1-\delta k}
\lesssim_{n,\delta} M^{-n}. \tag{$k := [\frac{n+1}{\delta}] + 1$}
\end{align*}
Now we integrate in $h_1$ and $h_2$, noting that the area of the domain of integration is $\pi$.

Turning to
\begin{align*}
\omega_3 &= -L_1\psi'(L_1g_1(x_2)x_1)\sin(Mg_2(x_2))\psi(L_2g_2(x_2))I,\\
I &= \int_{-\infty}^{x_3} \sin(Mg_3(x_2)s)\psi(L_3g_3(x_2)y)dy = \sum_{i=0}^k I_i + R_k,\\
I_i &= -\frac{\cos(Mg_3(x_2)x_3 + i\pi/2)L_3^i\psi^{(i)}(L_3g_3(x_2)x_3)}{M^{i+1}g_3(x_2)},\\
R_k &= \int_{-\infty}^{L_3x_3} \frac{\cos(ML_3^{-1}g_3(x_2)y + k\pi/2)L_3^k\psi^{(k+1)}(g_3(x_2)y)}{M^{k+1}}dy.
\end{align*}
Let $k = [n/\delta] + 1$. Since $L_3 \le M^{1-\delta}$,
\[
\|\text{part of $\omega_3$ involving $R_k$}\|_{C^0}
\lesssim_k L_1L_3^kM^{-k-1} \lesssim_{n,\delta} M^{-n}
\]
so, since supp $\omega_3 \subset \{|x| \le 1\}$ and the Biot--Savart kernel is locally integrable,
\[
\|(\text{part of $\omega_3$ involving $R_k$})e_3 * \mathcal K\|_{C^0}
\lesssim_{n,\delta} M^{-n}.
\]
Since the terms involving $I_i$ has similar estimates as $L_1L_3^iM^{-i-1}\omega_1$, and $L_1, L_3 \le M$,
their convolutions with the Biot--Savart kernel has the same bound as $\omega_1e_1 * \mathcal K$,
so does $\omega_3e_3 * \mathcal K$, and hence so does $\omega * \mathcal K$.

For $j > 0$, we fall the derivatives on the Biot--Savart kernel.
Since $|h| \ge M^{1-\delta}$, the bound is larger by a factor of $M^{j(1-\delta)}$,
so the result still holds.
\end{proof}

\begin{lem}\label{inner-interaction}
Let $\Omega$ be an axial vector field. Then
\begin{align*}
\|(\omega * \mathcal K) \cdot \nabla\Omega\|_{C^j}
&\lesssim_\delta (M + B_0)\sum_{j_1+j_2=j}M^{(j_1-2)(1-\delta)}\|\Omega\|_{C^{j_2+1}}. \tag{$j = 0, 1$}\\
\|\Omega \cdot \nabla(\omega * \mathcal K)\|_{C^s}
&\lesssim_\delta (M + B_0)M^{(s-1)(1-\delta)}D\|\Omega\|_{C^1}. \tag{$s \in [0, 1]$}
\end{align*}
\end{lem}
\begin{proof}
The first bound follows from $\|\omega * \mathcal K\|_{C^j} \lesssim_\delta (M + B_0)M^{(j-2)(1-\delta)}$ ($j = 0, 1$).

For the second bound, when $|x_i| > C/L_i$ for some $i = 1, 2, 3$, by Lemma \ref{farfield},
$\|\nabla(\omega * \mathcal K)\|_{C^1} \lesssim_\delta M^{-1}$
so $\|\Omega \cdot \nabla(\omega * \mathcal K)\|_{C^j}
\lesssim_\delta M^{-1}\|\Omega\|_{C^j}$ for $j = 0, 1$.
When $|x_i| \le C/L_i$ for all $i = 1, 2, 3$, since $\Omega$ is odd, $\Omega(0) = 0$, so
$\|\Omega\|_{C^j} \lesssim D^{1-j}\|\Omega\|_{C^1}$ for $j = 0, 1$.
Since $\|\nabla(\omega * \mathcal K)\|_{C^j} \lesssim_\delta (M + B_0)M^{(j-1)(1-\delta)}$,
\begin{align*}
\|\Omega \cdot \nabla(\omega * \mathcal K)\|_{C^j}
&\lesssim_\delta \sum_{j_1+j_2=j} (M + B_0)M^{(j_1-1)(1-\delta)}D^{1-j_2}\|\Omega\|_{C^1}\\
&\lesssim_\delta (M + B_0)M^{(j-1)(1-\delta)}D\|\Omega\|_{C^1}. \tag{$j = 0, 1$}
\end{align*}
because $D \gtrsim M^{-1}$. This dominates the previous bound
$\|\Omega \cdot \nabla(\omega * \mathcal K)\|_{C^j} \lesssim_\delta M^{-1}\|\Omega\|_{C^j}$.
The result follows from interpolation.
\end{proof}

\begin{lem}\label{outer-interaction}
Let $\Omega$ be an axial vector field. Let
\[
U(x_1,x_2,x_3) = (x_1\partial_1(\Omega * \mathcal K)_1(0, x_2, 0), (\Omega * \mathcal K)_2(0, x_2, 0), x_3\partial_3(\Omega * \mathcal K)_3(0, x_2, 0)).
\]
Then for $j = 0, 1$,
\begin{align*}
\|(\Omega * \mathcal K - U) \cdot \nabla\omega\|_{C^j}
&\lesssim_\delta (M^{j+1} + B_j)(1/L_1 + 1/L_3)^2D^{1-\delta}\|\Omega\|_{C^2},\\
\|\omega \cdot \nabla(\Omega * \mathcal K - U)\|_{C^j}
&\lesssim_\delta (M + B_0)^j(1/L_1 + 1/L_3)D^{1-\delta}\|\Omega\|_{C^2}.
\end{align*}
\end{lem}
\begin{proof}
Since $\Omega$ is axial, for $i = 1, 3$,
$(\Omega * \mathcal K - U)_i(\cdot, x_2, \cdot)$ vanishes to order 3 at (0, 0).
Then for $\beta = 1 - \delta \in (0, 1)$,
\[
|\nabla^2(\Omega * \mathcal K - U)_i(x_1, x_2, x_3)|
\le (|x_1|^\beta + |x_3|^\beta)\|\Omega * \mathcal K\|_{C^{2+\beta}}
\lesssim_\beta (|x_1|^\beta + |x_3|^\beta)\|\Omega\|_{C^2}
\]
so by Taylor's theorem with the Lagrange remainder, for $j = 0, 1, 2$,
\[
|\nabla^j(\Omega * \mathcal K - U)_i(x_1, x_2, x_3)|
\lesssim_\delta (|x_1|^{3-j+\delta} + |x_3|^{3-j+\delta})\|\Omega\|_{C^2}.
\]
For the second component, since $\Omega$ is odd, both $(\Omega * \mathcal K)_2$ and $U_2$ vanish when $x_2 = 0$,
so does $\nabla_{1,3}^2(\Omega * \mathcal K - U)_2$, so
\[
|\nabla_{1,3}^2(\Omega * \mathcal K - U)_2(x_1, x_2, x_3)|
\le |x_2|^{1-\delta}\|\Omega * \mathcal K\|_{C^{3-\delta}}
\lesssim_\delta |x_2|^{1-\delta}\|\Omega\|_{C^2}.
\]
Since $(\Omega * \mathcal K - U)_2(\cdot, x_2, \cdot)$ vanishes to order 2 at (0, 0), for $j = 0, 1, 2$,
\[
|\nabla^j(\Omega * \mathcal K - U)_2(x_1, x_2, x_3)|
\lesssim_\delta (|x_1|^{2-j} + |x_3|^{2-j})|x_2|^{1-\delta}\|\Omega\|_{C^2}.
\]
Since on the support of $\omega$, $|x_i| \lesssim 1/L_i$ ($i = 1, 3$) and $|x_2| \le D$,
\[
\|\Omega * \mathcal K - U\|_{C^j(\supp\omega)}
\lesssim_\delta (1/L_1 + 1/L_3)^{2-j}D^{1-\delta}\|\Omega\|_{C^2}.
\]
Since $\|\omega\|_{C^j} \lesssim (M + B_0)^j$, $\|\nabla\omega\|_{C^j} \lesssim M^{j+1} + B_j$ ($j = 0, 1$)
and $L_1, L_3 \le M$, the result follows.
\end{proof}

Finally we estimate the dissipation and control the error.
\begin{lem}\label{nabla-alpha}
For $\alpha \in (0, 1)$,
\[
||\nabla|^\alpha\omega_1 - M^\alpha(g_2'(x_2)^2 + g_3(x_2)^2)^{\alpha/2}\omega_1|
\lesssim_{\alpha,\delta} BM^{(\alpha-1)(1-\delta)}
\]
where $B = \sum_{i=1}^3L_i + L_3\ln M + \sup|a'| + \sup|g_1'| + ML_3^{-1}(\sup|g_3'| + \sup|g_2''|)$.
\end{lem}
\begin{proof}
Again we can replace all trignometric factors with exponential ones. We have
\[
|\nabla|^\alpha\omega_1 = c_\alpha\int \frac{\omega_1(x + h) - \omega_1(x)}{|h|^{3+\alpha}}dh,\quad
c_\alpha = \frac{2^\alpha\Gamma(\frac{\alpha+3}{2})}{\pi^{3/2}\Gamma(-\alpha/2)}.
\]
Let $\chi(h) = \chi(h_1, h_2)$ is a smooth bump function equal to 1 near the origin.
As in Lemma \ref{BS-C01} we integrate by parts in $h_3$ to get
\begin{align*}
&\left| \int_{\mathbb R} \frac{e^{iMg_3(x_2 + h_2)(x_3 + h_3)}\psi(L_3g_3(x_2 + h_2)(x_3 + h_3))}{|h|^{3+\alpha}}dh_3 \right|\\
&\lesssim_k M^{-k}\sum_{j=0}^k L_3^j(h_1^2 + h_2^2)^{-(2+\alpha+k-j)/2}
\end{align*}
which is the same as the corresponding bound in Lemma \ref{BS-C01} with $|J| = 1 + \alpha$,
so by the same reasoning,
\[
\int_{\mathbb R^2} (1 - \chi(M^{1-\delta}h))\int_{\mathbb R} \frac{\omega_1(x + h)}{|h|^{3+\alpha}}dh_3dh_1dh_2
\]
has the desired bound. Since $|h|^{-3-\alpha}$ is integrable at infinity,
\begin{align*}
&\int_{\mathbb R^2} (1 - \chi(M^{1-\delta}h))\int_{\mathbb R} \frac{a(x_2)\psi(L_1g_1(x_2)x_1)\psi(L_2g_2(x_2))\psi(L_3g_3(x_2)x_3)}
{|h|^{3+\alpha}}\\
&\times e^{iM(g_2(x_2)+g_2'(x_2)h_2+g_3(x_2)(x_3+h_3))}dh_3dh_1dh_2
\end{align*}
also has the same bound, so in $(1 - \chi(M^{1-\delta}h))(\omega_1(x + h) - \omega_1(x))$,
we can replace $\omega_1(x + h)$ with the planar wave
\[
a(x_2)\psi(L_1g_1(x_2)x_1)\psi(L_2g_2(x_2))\psi(L_3g_3(x_2)x_3)e^{iM(g_2(x_2)+g_2'(x_2)h_2+g_3(x_2)(x_3+h_3))}.
\]
In $\chi(M^{1-\delta}h)(\omega_1(x + h) - \omega_1(x))$,
we can also do so as in Lemma \ref{BS-C01},
with $|J|$ replaced by $1 + \alpha$, with acceptable error.
Now we end up with the action of $|\nabla|^\alpha$ on a factor independent of $h$ times
the planar wave $e^{iM(g_2'(x_2)h_2+g_3(x_2)h_3)}$,
whose result is multiplication by
$(g_2'(x_2)^2 + g_3(x_2)^2)^{\alpha/2}$; hence the result follows.
\end{proof}
\begin{cor}\label{nabla-alpha-0}
For $t \in [T_n, 1]$, $\alpha \in (0, 1)$ and $x \in \supp \omega_n(\cdot, t)$,
\begin{align*}
||\nabla|^\alpha\omega_1 - M^\alpha(g_2'(0)^2 + g_3(0)^2)^{\alpha/2}\omega_1|
&\lesssim_{\alpha,\delta} BM^{(\alpha-1)(1-\delta)}\\
&+ M^\alpha|x_2|\sup(|g_2''| + |g_3'|)
\end{align*}
where $B$ is the same as before.
\end{cor}
\begin{proof}
Since $|g_2'| \le 1$ and $|g_3| \le 2$,
$(g_2'^2 + g_3^2)^{\alpha/2}|_0^{x_{n+2}} \lesssim |x_2|\sup(|g_2''| + |g_3'|)$.
\end{proof}
\begin{cor}\label{nabla-alpha-Cs}
For $t \in [T_n, 1]$ and $0 < r < s < 1 - \alpha$,
\begin{align*}
\||\nabla|^\alpha\omega_1 - M^\alpha(g_2'(0)^2 + g_3(0)^2)^{\alpha/2}\omega_1\|_{C^r}
&\lesssim_{r,s,\alpha,\delta} BM^{\alpha+(s-1)(1-\delta)}\\
&+ M^{\alpha+s}|x_2|\sup(|g_2''| + |g_3'|).
\end{align*}
where $B$ is the same as before.
\end{cor}
\begin{proof}
Thanks to the embedding $L^\infty\cap|\nabla|^{-s}L^\infty \subset C^r$,
the result follows from the special case $r = 0$ shown above, and
\begin{align*}
&||\nabla|^{\alpha+s}\omega_1 - M^\alpha(g_2'(0)^2 + g_3(0)^2)^{\alpha/2}|\nabla|^s\omega_1|\\
&\le ||\nabla|^{\alpha+s}\omega_1 - M^{\alpha+s}(g_2'(0)^2 + g_3(0)^2)^{(\alpha+s)/2}\omega_1|\\
&+ M^\alpha(g_2'(0)^2 + g_3(0)^2)^{\alpha/2}
||\nabla|^s\omega_1 - M^s(g_2'(0)^2 + g_3(0)^2)^{s/2}\omega_1|\\
&\lesssim_{s,\alpha,\delta} BM^{(\alpha+s-1)(1-\delta)}
+ BM^{\alpha+(s-1)(1-\delta)}
+ M^{\alpha+s}|x_2|\sup(|g_2''| + |g_3'|)
\end{align*}
where the first term on the right-hand side are dominated by the second one,
and the corresponding bound with $s = 0$.
\end{proof}

\section{The construction}\label{construction}
Let $\alpha \in (0, \frac{22-8\sqrt7}{9})$.
Since $8\sqrt7 = \sqrt{448} > \sqrt{441} = 21$, $\alpha < 1/9$.
Let $R = \sqrt{\frac{2}{7\alpha}}$ and $s = (3\alpha+2-2\sqrt{14\alpha})/4$.
Then $R > \sqrt2$ and $s > 0$.
Let $A = N^{\sqrt{2\alpha/7}}$ and $L = N^{1+\alpha-s-2\sqrt{2\alpha/7}}$.
Then $L < N^{5/6}$. For $t \in [0, 1]$ let
\[
\omega_0(x, t) = A\psi(Lx_1)\psi(Lx_2)\psi(Lx_3)\sin(Nx_2)\sin(Nx_3)e_1 + \cdots e_3
\]
where the coefficient of $e_3$ is chosen to make $\omega_0$ divergence free,
see Lemma \ref{quadratic}.
Suppose $\omega_k(x, t)$ is defined for $0 \le k \le n - 1$ and $t \in [0, 1]$.
Let $\Omega_n = \sum_{k=0}^{n-1} \omega_k$ and
\begin{equation}\label{U-def}
\begin{aligned}
U_n(x, t) &= x_{n+1}\partial_{n+1}(\Omega_n * \mathcal K)_{n+1}(x_{n+2}e_{n+2}, t)e_{n+1}\\
&+ (\Omega_n * \mathcal K)_{n+2}(x_{n+2}e_{n+2}, t)e_{n+2}\\
&+ x_{n+3}\partial_{n+3}(\Omega_n * \mathcal K)_{n+3}(x_{n+2}e_{n+2}, t)e_{n+3}
\end{aligned}
\end{equation}
where the subscripts on the right-hand side are taken modulo 3.
Let $\phi_n$ be the flow of $U_n$.
We then define $\omega_n$ on $[T_n, 1]$ by solving the transport equation
\[
\partial_t\omega_n + U_n\cdot\nabla\omega_n = \omega_n\cdot\nabla U_n - g_n(t)\omega_n
\]
where
\[
g_n(t) = 0.01N^{\alpha R^n}(g_{n2}'(0)^2 + g_{n3}(0)^2)^{\alpha/2}
\]
with the data
\[
\omega_n(x, 1) = A^{R^n}\prod_{i=1}^3\psi(L^{R^n}x_i)\sin(N^{R^n}x_{n+2})\sin(N^{R^n}x_{n+3})e_{n+1}
+ \cdots e_{n+3}
\]
where the subscripts on the right-hand side are taken modulo 3,
and the coefficient of $e_{n+3}$ is chosen to make $\omega_n$ divergence-free,
see Lemma \ref{quadratic}.

\begin{lem}
For all $t \in [T_n, 1]$, the $e_{n+2}$ component of $\omega_n$ vanishes.
\end{lem}
\begin{proof}
Let $\omega$ be that component and $u$ be the $e_{n+2}$ component of $U_n$.
Then $\omega$ vanishes at time 1 and satisfies
\[
\partial_t\omega + U_n\cdot\nabla\omega = \omega_n\cdot\nabla u
= \omega\partial_{n+2}u - g_n(t)\omega
\]
because $u$ depends only on $x_{n+2}$.
This is a linear transport equation in terms of $\omega$,
so the lemma follows from the uniqueness theorem.
\end{proof}

For $i \not\equiv n + 2 \pmod 3$, the $i$-th component of the transport equation is
\begin{align*}
\partial_t(\omega_n)_i + U_n\cdot\nabla(\omega_n)_i
&= \omega_n\cdot\nabla(U_n)_i - g_n(t)\omega_n\\
&= (\omega_n)_i\partial_i(U_n)_i - g_n(t)\omega_n
\end{align*}
because the $i$-th component of $U_n$, depends only on $x_i$ and $x_{n+2}$,
and only the partial derivative in $x_i$ matters thanks to the lemma above.
Then, using the notation in Lemma \ref{3DPDE} and Lemma \ref{3DPDE2},
\begin{equation}\label{om-n-xt}
\begin{aligned}
\omega_n(x, t) &= G_n(t)K_n(t)^{-1}g_{n1}(x_{n+2})^{-1}\omega_n(\phi_n(x, t, 1), 1)_{n+1}e_{n+1}\\
&+ G_n(t)g_{n3}(x_{n+2})^{-1}\omega_n(\phi_n(x, t, 1), 1)_{n+3}e_{n+3}\\
&= G_n(t)K_n(t)^{-1}A^{R^n}\psi(L^{R^n}K_n(t)g_{n1}(x_{n+2})x_{n+1})\psi(L^{R^n}g_{n2}(x_{n+2}))\\
&\times\sin(N^{R^n}g_{n2}(x_{n+2}))\psi(L^{R^n}g_{n3}(x_{n+2})x_{n+3})\\
&\times\sin(N^{R^n}g_{n3}(x_{n+2})x_{n+3})e_{n+1} + \cdots e_{n+3}
\end{aligned}
\end{equation}
where
\begin{align*}
G_n(t) &= \exp\int_t^1 g_n(s)ds, &
K_n(t) &= \int_t^1 -\partial_{n+2}(U_n)_{n+2}(0, 0, 0, s)ds
\end{align*}
and the coefficient of $e_{n+3}$ makes $\omega_n(x, t)$ divergence free.

\subsection{The induction}
Note that $\omega_0$ is independent of $t$, so we have $G_0(t) = K_0(t) = g_{01} = g_{03} = g_{02}' = 1$
for all $t \in [0, 1]$.
Let $T_n^* = 34R^nA^{-R^{n-1}}\ln(AN^s)$.
Assume, as the induction hypothesis, that for all $0 \le k \le n - 1$,
and $t \in [T_{k+1}^*, 1]$ we have
$1 \le G_k(t) \le K_k(t) \le 2$,
$g_{k1}, g_{k3} \in [1/2, 2]$, $g_{k2}' \in [1/2, 1]$,
$|g_{k1}'|$, $|g_{k3}'|$ and $|g_{k2}''| < (N^2/L^R)^{R^{k-1}}$ and that
$|g_{k1}''|$ and $|g_{k3}''| < N^{2R^{k-1}}$. Then for $t \in [T_{k+1}^*, 1]$
we can apply the velocity estimates in Section \ref{om} to $\omega_k$,
with $M = N^{R^k}$, $L_2 = L_3 = L^{R^k}$ and $L_1 \le 2L^{R^k}$.

Also recall the definition of $U_n$ from \eqref{U-def}.
\begin{lem}\label{U-C1-bound}
If $N$ is sufficiently large then for $t \in [T_n^*, 1]$,
\[
-\partial_{n+2}(U_n)_{n+2}(0, 0, 0, t) \in (1/34, 35/34)A^{R^{n-1}}.
\]
\end{lem}
\begin{proof}
We have
\begin{align*}
B_0 &= \sum_{i=1}^3 L_i + L_3\ln M + \sup|g_{k1}'| + ML_3^{-1}(\sup|g_{k3}'| + \sup|g_{k2}''|)\\
&< L^{R^k}(4 + \ln M) + (1 + 2(N/L)^{R^k})(N^2/L^R)^{R^{k-1}}\\
&\lesssim L^{R^k}\ln M + (N^{1+2/R}/L^2)^{R^k}. \tag{$\ln M \ge \ln N \gg 1$, $N/L \gg 1$}
\end{align*}
We now check that $L > N^{(1+2/R)/3}$. Indeed, $L = N^{1+\alpha-s-2\sqrt{2\alpha/7}} = N^{(2+\alpha)/4+\frac32\sqrt{2\alpha/7}}
= N^{1/3+(2+3\alpha)/12+\frac32\sqrt{2\alpha/7}} > N^{1/3+\sqrt{14\alpha}/6+\frac32\sqrt{2\alpha/7}}
= \sqrt[3]N^{1+8\sqrt{2\alpha/7}} > \sqrt[3]N^{1+7\sqrt{2\alpha/7}} = N^{(1+2/R)/3}$.
Then $N^{1+2/R}/L^2 < L$, so the first term dominates.

Since $L_2 = L_3 = L^{R^k} \le M^{5/6}$, and $L_1 \le 2L^{R^k} \le 2M^{5/6}$,
we can apply Lemma \ref{BS-C01} and Corollary \ref{BS-C01-cor} to get,
for $t \in [T_{k+1}^*, 1]$ and $\delta \in (0, 1)$,
\[
\|\omega_k(t) * \mathcal K\|_{C^1} \lesssim_\delta A^{R^k}(1 + B_0M^{\delta-1}).
\]
Let $r = g_{k3}/g_{k2}'$. Then $r \in [1/2, 4]$, so
\[
\frac{G_k(t)g_{k2}'g_{k3}}{K_k(t)g_{k1}(g_{k3}^2 + g_{k2}'^2)}
\in \frac{[1/4, 2]}{r + 1/r} \in [1/17, 1]. \tag{$r = g_{k3}/g_{k2}' \in [1/2, 4]$}
\]
By Lemma \ref{BS-C1}, for $i \equiv k \pmod 3$,
\[
-\partial_i(\omega_k * \mathcal K)_i(0, 0, 0, t)
\in [1/17, 1]A^{R^k} + O_\delta(A^{R^k}B_0M^{\delta-1}).
\]
Since $B_0 \lesssim L^{R^k}\ln M \lesssim_\delta M^{5/6+\delta}$,
\[
-\partial_i(\omega_k * \mathcal K)_i(0, 0, 0, t)
\in [1/17, 1]A^{R^k} \pm O_\delta(A^{R^k}M^{2\delta-1/6}).
\]
Let $\delta \le 0.08$. For $k = n - 1$, $M^{2\delta-1/6} \le N^{2\delta-1/6}$ has a negative exponent,
so if $N$ is sufficiently large then for $t \in [T_n^*, 1]$,
\begin{equation}\label{U-C1-bound1}
-\partial_{n+2}(\omega_{n-1} * \mathcal K)_{n+2}(0, 0, 0, t) \in (3/68, 69/68)A^{R^{n-1}}.
\end{equation}
For $0 \le k \le n - 2$ and $t \in [T_{k+1}^*, 1]$,
\[
\|\omega_k(t) * \mathcal K\|_{C^1}
\lesssim A^{R^k}(1 + M^{2\delta-1/6}) \lesssim A^{R^k}
\]
so for $t \in [T_n^*, 1]$, if $A = N^{\sqrt{2\alpha/7}}$ is sufficiently large then
\begin{equation}\label{U-C1-bound2}
\sum_{k=0}^{n-2} \|\omega_k(t) * \mathcal K\|_{C^1}
\lesssim \sum_{k=0}^{n-2} A^{R^k}
\begin{cases}
\lesssim A^{R^{n-2}}, & n \ge 2\\
= 0, & n = 1
\end{cases}
< A^{R^{n-1}}/68.
\end{equation}
By \eqref{U-def}, $\partial_{n+2}(U_n)_{n+2}(0, 0, 0, t) = \partial_{n+2}(\Omega * \mathcal K)_{n+2}(0, 0, 0, t)$,
so the result follows from \eqref{U-C1-bound1} and \eqref{U-C1-bound2}.
\end{proof}

Next we want to find a subinterval $[T_n, 1] \subset [T_n^*, 1]$
in which the assumptions needed to apply the flow map estimates in Section \ref{flowmap} hold.
For that purpose we need to estimate higher derivatives.
\begin{lem}\label{U-C3-bound}
There is a constant $C > 1$ such that if $N$ is sufficiently large then for $t \in [T_n^*, 1]$ and $i = 1, 2, 3$,
$\|\partial_i(U_n)_i(t)\|_{C^2} \le C(AN^2)^{R^{n-1}}$.
\end{lem}
\begin{proof}
For $0 \le k \le n - 1$, there are two values of $i \in \{1, 2, 3\}$ such that
$\partial_i\omega_k$ is of the same form as $\omega_k$,
with an extra factor being $L_1g_{k1}(x_{k+2})$, $L_3g_{k3}(x_{k+2})$ or $Mg_{k3}(x_{k+2})$,
so for these two values of $i$ and $t \in [T_{k+1}^*, 1]$, by Corollary \ref{BS-C12-cor},
\[
\|\partial_i(\omega_k(t) * \mathcal K)_i\|_{C^2}
\lesssim_\delta A^{R^k}(M^2 + B_1M^\delta).
\]
Since $\omega_k * \mathcal K$ is divergence free, the above then holds for all $i = 1, 2, 3$.
Since $B_0 \lesssim L^{R^k}\ln M \lesssim_\delta M^{5/6+\delta} \le M$ for $M$ sufficiently large,
\begin{align*}
B_1 &= B_0(M + B_0) + \sup|g_{k1}''| + ML_3^{-1}\sup|g_{k3}''|\\
&\lesssim MB_0 + (1 + (N/L)^{R^k})N^{2R^{k-1}}
\lesssim ML^{R^k}\ln M + L^{2R^k} \lesssim ML^{R^k}\ln M
\end{align*}
where we have used the bound $N^{1+2/R}/L < L^2$ proven before.
Since $B_1 \lesssim_\delta M^{11/6+\delta}$ and $M \gg 1$,
for $t \in [T_{k+1}^*, 1]$ and $i = 1, 2, 3$,
\[
\|\partial_i(\omega_k(t) * \mathcal K)_i\|_{C^2}
\lesssim A^{R^k}M^2 = (AN^2)^{R^k}.
\]
Summing from 0 to $n - 1$ and noting that $\partial_i(U_n)_i(x, t) = \partial_i(\Omega_n * \mathcal K)_i(x_{n+2}e_{n+2}, t)$ gives the result.
\end{proof}

Let $X = N^{-R^{n-1}}P^{-R^n}$, where $P = 6C\ln A$, with $C$ given in Lemma \ref{U-C3-bound}.
We check that this choice of $X$ allows us to apply the lemmas in Section \ref{flowmap}.
\begin{lem}\label{X2UC3-le-U-C1}
If $A > e$ and $t \in [T_n^*, 1]$, then
$X^2\|\partial_{n+2}(U_n)_{n+2}(t)\|_{C^2} < -\partial_{n+2}(U_n)_{n+2}(0, 0, 0, t)$.
\end{lem}
\begin{proof}
Using $X^2 \le N^{-2R^{n-1}}P^{-2R^n}$ and Lemma \ref{U-C3-bound} we get
\begin{align}
\label{X2UC3-bound}
X^2\|\partial_{n+2}(U_n)_{n+2}(t)\|_{C^2} \le CA^{R^{n-1}}P^{-2R^n}.
\end{align}
Since $R$, $C$ and $\ln A > 1$, $P^{2R^n} > P > 36C$,
so the right-hand side is less than $A^{R^{n-1}}/36$.
Comparing this bound with the one given by Lemma \ref{U-C1-bound} shows the claim.
\end{proof}

By Lemma \ref{U-C1-bound}, there is $T_n$ such that $1 - T_n \in (34/35, 34)R^nA^{-R^{n-1}}\ln(AN^s)$
(in particular $T_n \in [T_n^*, 1]$) and
\[
\ln K_n(T_n) = \int_{T_n}^1 -\partial_{n+2}(U_n)_{n+2}(0, 0, 0, s)ds = R^n\ln(AN^s).
\]
Recall that $\phi_n$ is the flow map of the velocity field $U_n$ defined in \eqref{U-def}.
\begin{lem}\label{phin-bound}
If $N$ is sufficiently large (so is $A = N^{\sqrt{2\alpha/7}}$) then
for $x \in \supp \omega_n(\cdot, 1)$ and $t \in [T_n, 1]$,
$|\phi_n(x, 1, t)_{n+2}| \le X$.
\end{lem}
\begin{proof}
Since $K_n(t) \le K_n(T_n) = (AN^s)^{R^n}$,
we have $K_n(t)^{-1}X\ge (AN^{1/R+s}P)^{-R^n}$. By \eqref{X2UC3-bound},
\[
\int_t^1 \frac{X^2\|\partial_{n+2}(U_n)_{n+2}(s)\|_{C^2}}6ds
\le \frac{C(1 - T_n)}{6P^{2R^n}}A^{R^{n-1}}
< cR^n\ln(AN^s). \tag{$c = 6C/P^2 < 1/\ln A$}
\]
Then the exponential of the integral is less than $(A^cN^{cs})^{R^n}$, so
\[
K_n(t)^{-1}Xe^{\int_t^1 -\frac{X^2\|\partial_{n+2}(U_n)_{n+2}(s)\|_{C^2}}6ds}
> (A^{1+c}N^{1/R+s+cs}P)^{-R^n}.
\]
We now check that $L > AN^{1/R+s}$. Indeed, via the logarithm it is equivalent to
$1 + \alpha - s - 2\sqrt{2\alpha/7} < \sqrt{2\alpha/7} + \sqrt{7\alpha/2} + s$.
After substituting $s$ and rearranging, this is equivalent to $\sqrt{2\alpha/7} > \alpha$,
which is true because $\alpha < 2/7$. As $A \to \infty$, $c < 1/\ln A \to 0$,
so the above inequality (which is really between the exponents) can be improved to
$L > A^{1+c}N^{1/R+s+cs}$. Also thanks to the margin in the exponent,
together with the fact that $P \lesssim \ln A$, the inequality can be further improved to
$L > A^{1+c}N^{1/R+s+cs}P$. Then on supp $\omega_n(\cdot, 1)$ we have,
\[
|x_{n+2}| \le L^{-R^n} < K_n(t)^{-1}Xe^{\int_t^1 -\frac{X^2\|\partial_{n+2}(U_n)_{n+2}(s)\|_{C^2}}6ds}.
\]
Then the claim follows from Lemma \ref{1DODE2} (i).
\end{proof}

Next we record a useful limit.
\begin{lem}\label{rar->0}
If $r > 1$ then $ra^r$ and $ra^r\ln a \to 0$ as $a \to 0+$, uniformly in $r$.
\end{lem}
\begin{proof}
We have $\partial_r(ra^r) = (1 + r\ln a)a^r < 0$ when $r > 1/2$ and $a < 1/e^2$.
Then $ra^r < \sqrt a/2 \to 0$ as $a \to 0+$, uniformly in $r > 1/2$. Moreover,
$|\ln a| = \ln(1/a) = 2\ln(1/\sqrt a) < 2/\sqrt a < 2/\sqrt{a^r}$, so
$|ra^r\ln a| < 2ra^{r/2} = 4(r/2)a^{r/2} \to 0$ as $a \to 0+$, uniformly in $r > 1$.
\end{proof}

Now we close the induction hypothesis for $n$.
\begin{lem}\label{induction}
(i) For $t \in [T_{n+1}^*, 1]$, $1 \le K_n(t) < 2$.

(ii) For $t \in [T_n, 1]$, $g_{n1}$ and $g_{n3} \in (1/2, 2)$.

(iii) For $t \in [T_{n+1}^*, 1]$, $|g_{n1}'|$, $|g_{n3}'|$ and $|g_{n2}''| < (N^2/L^R)^{R^{n-1}}$.

(iv) For $t \in [T_{n+1}^*, 1]$, $|g_{n1}''|$ and $|g_{n3}''| < N^{2R^{n-1}}$.

(v) For $t \in [T_n, 1]$, $1 \le G_n(t) \le K_n(t)$.
\end{lem}
\begin{rem}
Recall that
\begin{align*}
1 - T_n &\in (34/35, 34)R^nA^{-R^{n-1}}\ln(AN^s),\\
1 - T_{n+1}^* &= 34R^{n+1}A^{-R^n}\ln(AN^s)
\end{align*}
so
\[
\frac{1 - T_{n+1}^*}{1 - T_n} < 35RA^{R^{n-1}-R^n} < 35RA^{1-R}
= 35R(A^{\frac{1-R}{R}})^R.
\]
Since $R > \sqrt2$, as $A \to +\infty$, $A^{\frac{1-R}{R}} \to 0+$, uniformly in $R$.
Then the right-hand side $\to 0$, also uniformly in $R$ by Lemma \ref{rar->0}.
In particular it can be made smaller than 1. Then $T_{n+1}^* > T_n$,
so (ii) and (v) also hold on $[T_{n+1}^*, 1]$, the time interval needed to recover the induction hypothesis.
\end{rem}
\begin{proof}
(i) For $t \in [T_{n+1}^*, 1]$ we have
\begin{align*}
0 &\le \ln K_n(t) = \int_t^1 -\partial_{n+2}(U_n)_{n+2}(0, 0, 0, s)ds \lesssim (1 - T_{n+1})A^{R^{n-1}}\\
&\lesssim R^{n+1}A^{R^{n-1}-R^n}\ln(AN^s) \lesssim R^{n+1}(A^{\frac{1-R}{R^2}})^{R^{n+1}}\ln A.
\end{align*}
Since $R > \sqrt2$, as $A \to +\infty$, $A^{\frac{1-R}{R^2}} \to 0+$, uniformly in $R$.
Then the right-hand side $\to 0$, also uniformly in $R$ by Lemma \ref{rar->0}.
In particular it can be made smaller than ln2.

(ii) By Lemma \ref{BS-C1}, for $0 \le k \le n - 1$ and $t \in [T_{k+1}^*, 1]$,
\[
|\nabla(\omega_k(t) * \mathcal K)_{k+1}|
\lesssim_\delta A^{R^k}L_1(M + B_0)M^{\delta-2}
\lesssim A^{R^k}L_1M^{\delta-1}.
\]
Since $AL_1 \le 2AL = N^{1+\alpha-s-\sqrt{2\alpha/7}} = N^{(2+\alpha)/4+2.5\sqrt{2\alpha/7}}$
with an exponent less than 1, taking $\delta$ sufficiently small we have that
$|\nabla(\omega_k(t) * \mathcal K)_{k+1}| \lesssim 1$.
In particular, this holds for $\partial_n(\omega_{n-1}(t) * \mathcal K)_n$
on the time interval $[T_n^*, 1]$. On the other hand,
\[
\sum_{k=0}^{n-2} \|\omega_k * \mathcal K\|_{C^1}
\lesssim \sum_{k=0}^{n-2} A^{R^k} \lesssim A^{R^{n-2}}
\]
so for $t \in [T_n^*, 1]$,
\[
|\partial_n(U_n(t))_n| = |\partial_n(\Omega_n(t) * \mathcal K)_n| \lesssim A^{R^{n-2}},
\]
\[
\int_t^1 |\partial_n(U_n)_n(s)|ds
\lesssim R^nA^{R^{n-2}-R^{n-1}}\ln(AN^s)\\
\lesssim R^n(A^{\frac{1-R}{R^2}})^{R^n}\ln A \to 0
\]
as $A \to \infty$, uniformly in $R$ and $n$ by Lemma \ref{rar->0}.
Also by \eqref{X2UC3-bound},
\begin{equation}\label{int-X2UC3->0}
\begin{aligned}
\int_t^1 \frac{X^2\|\partial_{n+2}(U_n)_{n+2}(s)\|_{C^2}}2ds
&\lesssim \frac{1 - T_n}{P^{2R^n}}A^{R^{n-1}}
\lesssim \frac{R^n\ln(AN^s)}{P^{2R^n}}\\
&\lesssim R^n(\ln A)^{1-2R^n}
\le R^n(\ln A)^{-R^n} \to 0
\end{aligned}
\end{equation}
as $A \to \infty$, uniformly in $R$ and $n$ by Lemma \ref{rar->0}.
Then by Lemma \ref{3DPDE}, when $A$ is sufficiently large,
$|\ln g_{n3}|$ and $|\ln g_{n1}| < \ln2$ for $t \in [T_n, 1]$,

(iii) By Lemma \ref{1DODE} (iv), for $t \in [T_{n+1}^*, 1]$ and $|x_2| \le X$,
\[
|\phi_n(x, t, 1)_{n+2}|
\ge K_n(t)^{-1}|x_{n+2}|\exp\int_t^1 -\frac{X^2\|\partial_{n+2}(U_n)_{n+2}(s)\|_{C^2}}6ds
\gtrsim |x_{n+2}|
\]
so for $x \in \supp\omega_n(\cdot, 1)$, $|\phi_n(x, 1, t)_{n+2}| \lesssim |x_{n+2}| \lesssim L^{R^{-n}}$.
Then by Lemma \ref{U-C3-bound}, for $i = 1, 2, 3$,
\[
|\phi_n(x, 1, t)_{n+2}|\|\partial_i(U_n)_i(t)\|_{C^2} \lesssim (AN^2/L^R)^{R^{n-1}}
\]
so for $t \in [T_{n+1}^*, 1]$,
\begin{align*}
|\phi_n(x, 1, t)_{n+2}|\int_t^1 \|\partial_i(U_n)_i(s)\|_{C^2}ds
&\lesssim R^{n+1}A^{R^{n-1}-R^n}\ln(AN^s)(N^2/L^R)^{R^{n-1}}\\
&< (N^2/L^R)^{R^{n-1}}/4
\end{align*}
when $A = N^{\sqrt{2\alpha/7}}$ is sufficiently large.
Then by Lemmas \ref{1DODE2} and \ref{3DPDE},
$|g_{n2}''|$, $|(\ln g_{n3})'|$ and $|(\ln g_{n1})'| < (N^2/L^R)^{R^{n-1}}/2$.
Since $|g_{n1}|$, $|g_{n3}| < 2$, we have $|g_{n3}'|$,
$|g_{n1}'| < (N^2/L^R)^{R^{n-1}}$ for $t \in [T_{n+1}^*, 1]$.

(iv) Similarly for $t \in [T_{n+1}^*, 1]$ and $i = 1, 2, 3$,
when $A = N^{\sqrt{2\alpha/7}}$ is sufficiently large,
\[
\int_t^1 \|\partial_i(U_n)_i(s)\|_{C^2}ds
\lesssim R^{n+1}A^{R^{n-1}-R^n}\ln(AN^s)N^{2R^{n-1}} < N^{2R^{n-1}}/8.
\]
Since $K_n(t) \le 2$, $1 + \ln K_n(t) \le 2$,
so by Lemma \ref{3DPDE}, $|(\ln g_{n3})''|$ and $|(\ln g_{n1})''| < N^{2R^{n-1}}/4$.
Then for $i = 1, 3$,
\[
|g_{ni}''/g_{ni}| \le |(\ln g_{ni})''| + (g_{ni}'/g_{ni})^2
= |(\ln g_{ni})''| + ((\ln g_{ni})')^2 < N^{2R^{n-1}}/2.
\]
Since $|g_{n1}|$ and $|g_{n3}| < 2$, we have $|g_{n3}''|$ and $|g_{n1}''| < N^{2R^{n-1}}$ for $t \in [T_{n+1}^*, 1]$.

(v) For $t \in [T_n, 1]$ we have $g_n(t) = 0.01N^{\alpha R^n}(g_{n2}'(0)^2 + g_{n3}(0)^2)^{\alpha/2} \ge 0$,
$g_n(t) < 0.01\sqrt[18]{5}N^{\alpha R^n} < N^{\alpha R^n}/68 = A^{R^{n-1}}/68 < -\partial_{n+2}(\Omega * \mathcal K)_{n+2}(0, 0, 0, t)/2$,
so integrating gives $1 \le G_n(t) \le K_n(t)$.
\end{proof}

\subsection{Estimating $\omega_n$}
By Lemma \ref{induction} (ii), $|\ln g_{n1}|$ and $|\ln g_{n3}| < \ln2$ for $t \in [T_n, 1]$.
Since $K_n(t) \le K_n(T_n) = (AN^s)^{R^n}$,
by Lemma \ref{1DODE} (iv) and \eqref{int-X2UC3->0}, for $|x_2| \le X$,
\begin{align*}
|\phi_n(x, t, 1)_{n+2}|
&\ge K_n(t)^{-1}|x_{n+2}|\exp\int_t^1 -\frac{X^2\|\partial_{n+2}(U_n)_{n+2}(s)\|_{C^2}}6ds\\
&\gtrsim K_n(t)^{-1}|x_{n+2}|
\end{align*}
so for $x \in \supp\omega_n(\cdot, 1)$, $|\phi_n(x, 1, t)_{n+2}| \lesssim K_n(t)L^{-R^n}$.
Then by Lemma \ref{U-C3-bound}, for $i = 1, 2, 3$,
\[
|\phi_n(x, 1, t)_{n+2}|\|\partial_i(U_n)_i(t)\|_{C^2} \lesssim K_n(t)(AN^2/L^R)^{R^{n-1}}.
\]
Since $1 - T_n \lesssim R^nA^{-R^{n-1}}\ln(AN^s) \lesssim R^nA^{-R^{n-1}}\ln N$,
\[
|\phi_n(x, 1, t)_{n+2}|\int_t^1 \|\partial_i(U_n)_i(s)\|_{C^2}ds
\lesssim K_n(t)R^n(N^2/L^R)^{R^{n-1}}\ln N.
\]
Since $X^2\|\partial_{n+2}(U_n)_{n+2}(t)\|_{C^2} < -\partial_{n+2}(U_n)_{n+2}(0, 0, 0, t)$ and
$|\phi_n(x, 1, t)_{n+2}| < X$, Lemma \ref{1DODE2} and \ref{3DPDE} apply to give
\[
|g_{n2}''|, |(\ln g_{n1})'|, |(\ln g_{n3})'|
\lesssim K_n(t)R^n(N^2/L^R)^{R^{n-1}}\ln N.
\]
Since $|g_{n1}|, |g_{n3}| \le 2$, $|g_{n1}'|$ and $|g_{n3}'|$ have the same bound.
Also, for $i = 1, 2, 3$,
\[
\int_t^1 \|\partial_i(U_n)_i(s)\|_{C^2}ds
\lesssim R^nN^{2R^{n-1}}\ln N
\]
so, using Lemma \ref{3DPDE} and $K_n(t) \le (AN^s)^{R^n} = N^{(\sqrt{2\alpha/7}+s)R^n} < N^{(3\alpha+2)R^n/4} < N^{R^n}$,
\[
|(\ln g_{n1})''|, |(\ln g_{n3})''|
\lesssim (1 + \ln K_n(t))R^nN^{2R^{n-1}}\ln N
\lesssim R^nN^{2R^{n-1}}\ln^2N.
\]

When $t \in [T_n, 1]$, for $\omega_n$ we have $M = N^{R^n}$, $L_2 = L_3 = L^{R^n}$, $L_1 = L^{R^n}K_n(t)$.
Then
\begin{align*}
B_0 &= \sum_{i=1}^3 L_i + L_3\ln M + \sup|g_1'| + ML_3^{-1}(\sup|g_3'| + \sup|g_2''|)\\
&\lesssim K_n(t)L^{R^n}\ln M + K_n(t)(N^{1+2/R}/L^2)^{R^n}\ln M\\
&\lesssim K_n(t)L^{R^n}\ln M
\end{align*}
because the first term dominates as before.

Since $L = N^{1+\alpha-s-2\sqrt{2\alpha/7}} = N^{1-s+\alpha-\sqrt{2\alpha/7}}/A$, we have
\[
K_n(t)L^{R^n} \le (ALN^s)^{R^n} = N^{(1+\alpha-\sqrt{2\alpha/7})R^n} = M^{1+\alpha-\sqrt{2\alpha/7}}.
\]
Moreover $\alpha < 2/7$, so Lemma \ref{BS-C01} and Lemma \ref{BS-C12} apply and $B_0 \lesssim M$. Then
\begin{align*}
B_1 &= B_0(M + B_0) + \sup|g_1''| + ML_3^{-1}\sup|g_3''|\\
&\lesssim MB_0 + (N/L)^{R^n}N^{2R^{n-1}}\ln^2M\\
&\lesssim K_n(t)ML^{R^n}\ln M + L^{2R^n}\ln^2M \le 2K_n(t)ML^{R^n}\ln M
\le 2M^2\ln M
\end{align*}
because $K_n(t) \ge 1$ and $L^{R^n}\ln^2M < M^{5/6}\ln^2M < M$ when $M \gg 1$.

\subsection{Estimating the force}
In our construction, the vorticity is an infinite sum whose terms are turned on at different time scales:
\[
\omega(x, t) = \sum_{n=0}^\infty \rho_n(t)\omega_n(x, t).
\]
where $\rho_n$ is supported on $[T_n, 1]$, has values in [0, 1], and is a smooth approximation of $1_{[T_n,1]}$ in $L^1$.
By \eqref{om-n-xt}, each $\omega_n$ is smooth in $(x, t)$. Since
\[
1 - T_n < 34R^nA^{-R^{n-1}}\ln(AN^s) \to 0 \tag{$n \to \infty$}
\]
on any compact subset of $[0, 1)$, only finitely many $\omega_n$ is non-zero,
so the whole sum is smooth in $(x, t) \in \mathbb R^3 \times [0, 1)$, with
\[
\partial_t\omega(x, t) = \sum_{n=0}^\infty (\rho_n'(t)\omega_n(x, t) + \rho_n(t)\partial_t\omega_n(x, t))
\]
so
\[
\partial_t\omega + (\omega * \mathcal K)\cdot\nabla\omega
- \omega\cdot\nabla(\omega * \mathcal K) - 0.01|\nabla|^\alpha\omega
= \sum_{n=0}^\infty F_n
\]
where the sum is smooth in $(x, t) \in \mathbb R^3 \times [0, 1)$, and the summand
\begin{align*}
F_n &= \rho_n(t)F_n^\circ(x, t) + \rho_n'(t)\omega_n(x, t),\\
F_n^\circ &= (\omega_n * \mathcal K)\cdot\nabla\omega_n
- \omega_n\cdot\nabla(\omega_n * \mathcal K)\\
&+ (\omega_n * \mathcal K)\cdot\nabla\Omega_n
- \Omega_n\cdot\nabla(\omega_n * \mathcal K)\\
&+ (\Omega_n * \mathcal K - U_n)\cdot\nabla\omega_n
- \omega_n\cdot\nabla(\Omega_n * \mathcal K - U_n)\\
&+ (g_n(t)\omega_n - 0.01|\nabla|^\alpha)\omega_n\\
&+ \partial_t\omega_n + U_n\cdot\nabla\omega_n - \omega_n\cdot\nabla U_n - g_n(t)\omega_n
\tag{= 0 by construction}
\end{align*}
will be controlled in the $C^r$ norm, where $r \in (0, s)$,
which is possible because $s > 0$:
\[
\|F_n\|_{L_t^1C_x^r} \le \|\rho_n(t)F_n^\circ\|_{L_t^1C_x^r} + \|\rho_n'(t)\omega_n\|_{L_t^1C_x^r}.
\]
Since $\rho_n$ is supported on $[T_n, 1]$ and has values in [0, 1],
\[
\|\rho_n(t)F_n^\circ\|_{L_t^1C_x^r} \le \int_{T_n}^1 \|F_n^\circ(t)\|_{C^r}dt.
\]
Since $\omega_n$ is continuous in $t \in [T_n, 1]$,
\[
\|\rho_n'(t)\omega_n\|_{L_t^1C_x^r} \to \|\omega_n(\cdot, T_n)\|_{C^r}
\]
as $\rho_n$ is taken to be arbitrarily close to $1_{[T_n,1]}$ in $L^1$.
Thus to show that $\sum_{n=0}^\infty F_n$ is finite in $L_t^1C_x^r$ norm,
it suffices to show the finiteness of the sums $\sum_{n=0}^\infty \|\omega_n(\cdot, T_n)\|_{C^r}$ and $\sum_{n=0}^\infty \int_{T_n}^1 \|F_n^\circ(t)\|_{C^r}dt$.
The former is shown in Section \ref{wn-Tn}, while the latter, itself containing multiple terms, are shown in Sections \ref{self-interact} to \ref{dissipate}.

\subsubsection{The term $\omega_n(T_n)$}\label{wn-Tn}
For $t \in [T_n, 1]$ we have $\sup|\omega_n(\cdot, t)| \lesssim K_n(t)^{-1}A^{R^n}$ and
\[
\|\omega_n(\cdot, t)\\|_{C^1} \lesssim K_n(t)^{-1}A^{R^n}(M + B_0)
\lesssim K_n(t)^{-1}A^{R^n}M
\]
so for $r \in (0, s)$, $\|\omega_n\|_{C^r} \lesssim K_n(t)^{-1}A^{R^n}M^r$.
In particular, since
\[
K_n(T_n) = \exp\int_{T_n}^1 -\partial_{n+2}U_{n+2}(0, 0, 0, s)ds = (AN^s)^{R^n} = A^{R^n}M^s,
\]
$\|\omega_n(T_n)\|_{C^r} \lesssim M^{r-s} = N^{(r-s)R^n}$,
so, using $r - s < 0$,
\[
\sum_{n=0}^\infty \|\omega_n(\cdot, T_n)\|_{C^r} \lesssim \infty.
\]

\subsubsection{The self-interaction}\label{self-interact}
By Lemma \ref{quadratic}, for $r \in (0, s)$,
\begin{align*}
\|(\omega_n * \mathcal K) \cdot \nabla\omega_n\|_{C^r}
&+ \|\omega_n \cdot \nabla(\omega_n * \mathcal K)\|_{C^r}\\
&\lesssim_\delta K_n(t)^{-2}A^{2R^n}B_0^{1-r}B_1^rM^{-2+2\delta}(M + B_0)\\
&\lesssim K_n(t)^{-2}A^{2R^n}B_0^{1-r}B_1^rM^{-1+2\delta}\\
&\lesssim A^{2R^n}L^{R^n}M^{r-1+2\delta}\ln M
\end{align*}
where in the last line we used $B_j \lesssim K_n(t)M^jL^{R^n}\ln M$ ($j = 0, 1$), so
\begin{align*}
\int_{T_n}^1 (\|(\omega_n * \mathcal K) \cdot \nabla\omega_n\|_{C^r}
&+ \|\omega_n \cdot \nabla(\omega_n * \mathcal K)\|_{C^r})dt\\
&\lesssim_\delta (1 - T_n)A^{2R^n}L^{R^n}M^{r-1+2\delta}\ln M\\
&\lesssim A^{2R^n-R^{n-1}}L^{R^n}M^{r-1+2\delta}\ln^2M\\
&\lesssim_\delta A^{2R^n-R^{n-1}}L^{R^n}M^{r-1+3\delta}\\
&= (N^{2\sqrt{2\alpha/7}-\alpha+r-1+3\delta}L)^{R^n}
= N^{(r+3\delta-s)R^n}
\end{align*}
because $L = N^{1+\alpha-s-2\sqrt{2\alpha/7}}$. Since $r < s$,
we can let $\delta = (s - r)/4 > 0$ so that
the integral is $O(a^{R^n})$ for some $a \in (0, 1)$, so
\[
\sum_{n=0}^\infty \int_{T_n}^1 (\|(\omega_n * \mathcal K) \cdot \nabla\omega_n\|_{C^r}
+ \|\omega_n \cdot \nabla(\omega_n * \mathcal K)\|_{C^r})dt < \infty.
\]

\subsubsection{The interaction between inner velocity and outer vorticity}
For $t \in [T_n, 1]$ and $k = 0, \dots, n - 1$, $K_k(t) \in [1, 2]$, so
$\|\omega_k\|_{C^{j}} \lesssim (AN^j)^{R^k}$ ($j = 0, 1, 2$).
Summing over $k$ we get $\|\Omega_n\|_{C^{j}} \lesssim (AN^j)^{R^{n-1}}$ ($j = 0, 1, 2$).
Then by Lemma \ref{inner-interaction}, for $j = 0, 1$,
\begin{align*}
\|(\omega_n * \mathcal K) \cdot \nabla\Omega_n\|_{C^j}
&\lesssim_\delta K_n(t)^{-1}A^{R^n}(M + B_0)\sum_{j_1+j_2=j}M^{(j_1-2)(1-\delta)}\|\Omega_n\|_{C^{j_2+1}}\\
&\lesssim_\delta A^{R^n}M\sum_{j_1+j_2=j}M^{(j_1-2)(1-\delta)}(AN^{j_2+1})^{R^{n-1}}
\end{align*}
where we have used $B_0 \lesssim M$ and $K_n(t) \ge 1$.
Since $N^{R^{n-1}} = M^{1/R} \le M^{1-\delta}$,
the sum is dominated by the summand with $j_1 = j$ and $j_2 = 0$, so
\[
\|(\omega_n * \mathcal K) \cdot \nabla\Omega_n\|_{C^j}
\lesssim_\delta M^{1+(j-2)(1-\delta)}(A^{1+R}N)^{R^{n-1}}.
\]
By interpolation, for $r \in (0, s)$,
\[
\|(\omega_n * \mathcal K) \cdot \nabla\Omega_n\|_{C^r}
\lesssim_\delta M^{1+(r-2)(1-\delta)}(A^{1+R}N)^{R^{n-1}}.
\]

Since for $t \in [T_n, 1]$, $g_{n1}$ and $g_{n3} > 1/2$, for $x \in \supp\omega_n$ we have
$|x_{n+1}|$ and $|x_{n+3}| < 2L^{-R^n}$.
Also $|x_2| = |\phi_n(\phi_n(x, t, 1), 1, t)_{n+2}| \lesssim K_n(t)L^{-R^n}$,
so $D = \max\{|x|: x \in \supp\omega_n(\cdot, t)\} \lesssim K_n(t)L^{-R^n}$.
Then by Lemma \ref{inner-interaction},
\begin{align*}
\|\Omega_n \cdot \nabla(\omega_n * \mathcal K)\|_{C^r}
&\lesssim_\delta K_n(t)^{-1}A^{R^n}M^{1+(r-1)(1-\delta)}D\|\Omega_n\|_{C^1}\\
&\lesssim A^{R^n}M^{1+(r-1)(1-\delta)}(1/L^{R^n})(AN)^{R^{n-1}}\ln M\\
&\lesssim_\delta M^{r+2\delta}L^{-R^n}(A^{1+R}N)^{R^{n-1}}
\end{align*}
dominating the bound above because $L \le N$. Then
\begin{align*}
\int_{T_n}^1 (\|(\omega_n * \mathcal K) \cdot \nabla\Omega_n\|_{C^r}
&+ \|\Omega_n \cdot \nabla(\omega_n * \mathcal K)\|_{C^r})dt\\
&\lesssim_\delta (1 - T_n)M^{r+2\delta}L^{-R^n}(A^{1+R}N)^{R^{n-1}}\\
&\lesssim M^{r+2\delta}(A/L)^{R^n}N^{R^{n-1}}\ln M\\
&\lesssim_\delta M^{r+3\delta}(A/L)^{R^n}N^{R^{n-1}} < M^{r-s+3\delta}
\end{align*}
because $L > AN^{1/R+s}$.
Since $r < s$, we can choose $\delta > 0$ so that $r - s + 3\delta < 0$, and then
\[
\sum_{n=0}^\infty \int_{T_n}^1 (\|(\omega_n * \mathcal K) \cdot \nabla\Omega_n\|_{C^r}
+ \|\Omega_n \cdot \nabla(\omega_n * \mathcal K)\|_{C^r})dt < \infty.
\]

\subsubsection{The interaction between outer velocity and inner vorticity}
By Lemma \ref{outer-interaction},
\begin{align*}
\|(\Omega_n * \mathcal K - U_n) \cdot \nabla\omega_n\|_{C^r}
&+ \|\omega_n \cdot \nabla(\Omega_n * \mathcal K - U_n)\|_{C^r}\\
&\lesssim_\delta K_n(t)^{-1}A^{R^n}(M^{r+1} + B_0^{1-r}B_1^r)\\
&\times (1/L_1 + 1/L_3)^2D^{1-\delta}\|\Omega\|_{C^2}\\
&\lesssim A^{R^n}M^{r+1}L^{-(3-\delta)R^n}(AN^2)^{R^{n-1}}.
\end{align*}
Since $1 - T_n \lesssim A^{-R^{n-1}}\ln M$,
\begin{align*}
\int_{T_n}^1 (\|(\Omega_n * \mathcal K - U_n) \cdot \nabla\omega_n\|_{C^r}
&+ \|\omega_n \cdot \nabla(\Omega_n * \mathcal K - U_n)\|_{C^r})dt\\
&\lesssim A^{R^n}M^{r+1}L^{-(3-\delta)R^n}N^{2R^{n-1}}\ln M\\
&= (N^{8\sqrt{2\alpha/7}+r+1}L^{-(3-\delta)})^{R^n}\ln M.
\end{align*}
When $\delta \to 0$, the base tends to
\begin{align*}
N^{8\sqrt{2\alpha/7}+r+1}L^{-3}
&= N^{8\sqrt{2\alpha/7}+r+1-3(1+\alpha-s-2\sqrt{2\alpha/7})}\\
&= N^{2\sqrt{14\alpha}-2-3\alpha+r+3s}.
\end{align*}
Since $r \in (0, s)$, $r + 3s < 4s = 3\alpha + 2 - 2\sqrt{14\alpha}$,
so the exponent is negative, so as before
\[
\sum_{n=0}^\infty \int_{T_n}^1 (\|(\Omega_n * \mathcal K - U_n) \cdot \nabla\omega_n\|_{C^r}
+ \|\omega_n \cdot \nabla(\Omega_n * \mathcal K - U_n)\|_{C^r})dt < \infty.
\]

\subsubsection{The dissipation}\label{dissipate}
For $x \in \supp\omega_n(\cdot, t)$, we have $\phi_n(x, t, 1) \in \supp\omega_n(\cdot, 1)$,
so
\[
|x_{n+2}| = |\phi_n(\phi_n(x, t, 1), 1, t)_{n+2}| \lesssim K_n(t)L^{-R^n}.
\]
Also $|g_{n2}''|$ and $|g_{n3}'| \lesssim K_n(t)(N^2/L^R)^{R^{n-1}}\ln M$.
Then by Corollary \ref{nabla-alpha-Cs},
\begin{align*}
\||\nabla|^{\alpha}\omega_1 &- M^\alpha(g_2'(0)^2 + g_3(0)^2)^{\alpha/2}|\nabla|\omega_1\|_{C^r}\\
&\lesssim_{r,s,\alpha,\delta} BM^{\alpha+(s-1)(1-\delta)}
+ M^{\alpha+s}|x_2|\sup(|g_2''| + |g_3'|)\\
&\lesssim BM^{\alpha+(s-1)(1-\delta)}
+ M^{\alpha+s}K_n(t)^2(N/L^R)^{2R^{n-1}}\ln M
\end{align*}
Now we specialize to $a = 1/g_1$, so $\sup|a'| \lesssim \sup|g_1'|$,
and the respective bound has $B$ replaced by $B_0 = B$ without the term $\sup|a'|$.
Since $B_0 \lesssim K_n(t)L^{R^n}\ln M$,
\begin{align*}
\||\nabla|^\alpha\omega_1 &- M^\alpha(g_2'(0)^2 + g_3(0)^2)^{\alpha/2}\omega_1\|_{C^r}\\
&\lesssim_{r,s,\alpha,\delta} K_n(t)L^{R^n}(\ln M)M^{\alpha+(s-1)(1-\delta)}\\
&+ M^{\alpha+(r+s)/2}K_n(t)^2(N/L^R)^{2R^{n-1}}\ln M.
\end{align*}
As before the corresponding bound for $\omega_3$ is better, so for $t \in [T_n, 1]$,
taking into account the amplitude of $\omega_n$,
\begin{align*}
\|0.01|\nabla|^\alpha\omega_n &- g_n(t)\omega_n\|_{C^r}\\
&\lesssim_{r,s,\alpha,\delta} (AL)^{R^n}(\ln M)M^{\alpha+(s-1)(1-\delta)}\\
&+ A^{2R^n}M^{\alpha+(r+3s)/2}(N/L^R)^{2R^{n-1}}\ln M \tag{$K_n(t) \le (AN^s)^{R^n}$}\\
&= M^{\sqrt{2\alpha/7}+1+\alpha-s-2\sqrt{2\alpha/7}+\alpha+(s-1)(1-\delta)}\ln M\\
&+ M^{2\sqrt{2\alpha/7}+\alpha+(r+3s)/2+2\sqrt{7\alpha/2}-2(1+\alpha-s-2\sqrt{2\alpha/7})}\ln M\\
&= M^{2\alpha-\sqrt{2\alpha/7}+(1-s)\delta}\ln M.
\end{align*}
Since $1 - T_n \lesssim A^{-R^{n-1}}\ln M = M^{-\alpha}\ln M$,
\begin{align*}
\int_{T_n}^1 \|0.01|\nabla|^\alpha\omega_n - g_n(t)\omega_n\|_{C^r}dt
&\lesssim_{r,s,\alpha,\delta} M^{\alpha-\sqrt{2\alpha/7}+(1-s)\delta}\ln M.
\end{align*}
Since $\alpha < 2/7$, when $\delta$ is sufficiently small, the exponent is negative, so
\[
\sum_{n=0}^\infty \int_{T_n}^1 \|0.01|\nabla|^\alpha\omega_n - g_n(t)\omega_n\|_{C^r}dt < \infty.
\]

Combining the bounds above, we get
\[
\sum_{n=0}^\infty \|F_n\|_{L_1^tC_x^r}
\le \sum_{n=0}^\infty \int_{T_n}^1 \|G_n(t)\|_{C^r}dt + \sum_{n=0}^\infty \|w_n(T_n)\|_{C^r}
< \infty.
\]

\subsection{Proof of Theorem \ref{thm}}
In the above, for $0 \le \alpha < \frac19(22 - 8\sqrt7)$ we have constructed a divergence free
$\omega \in C^\infty(\mathbb R^3 \times [0, 1))$ supported in $B(1) \times [0, 1)$ such that
\[
\partial_t\omega + u \cdot \nabla\omega + \nu|\nabla|^\alpha\omega = \omega \cdot \nabla u + F
\]
where $\nu = 0.01$, $F \in C^\infty(\mathbb R^3 \times [0, 1)) \cap L_t^1([0,T])C_x^r$ for some $r > 0$, and
\[
u(x, t) = \frac1{4\pi}\int_{\mathbb R^3} \frac{\omega(y, t) \times (x - y)}{|x - y|^3}dy
\]
satisfies $\nabla \cdot u = 0$ and $\nabla \times u = \omega$. Now we do a Hodge decomposition
\[
\partial_tu + u \cdot \nabla u + \nu|\nabla|^\alpha u = \nabla p + f
\]
where $\nabla \cdot f = 0$. Then
\[
\partial_t\omega + u \cdot \nabla\omega + \nu|\nabla|^\alpha\omega = \omega \cdot \nabla u + \nabla \times f
\]
so $\nabla \times f = F$. Since $\nabla \cdot f = 0$, $f$ can be recovered from $F$ using the Biot--Savart law:
\[
f(x, t) = \frac1{4\pi}\int_{\mathbb R^3} \frac{F(y, t) \times (x - y)}{|x - y|^3}dy.
\]
Since $F$ is compactly supported in $\mathbb R^3 \times [0, 1]$ and smooth in $x$ for $t < 1$,
$f$ is square integrable and smooth in $x$ for $t < 1$.
Since $F \in L_t^1([0,1])C_x^r$, by Schauder estimates $f \in L_t^1([0,1])C_x^{1,r}$.

\section*{Acknowledgements}
This work is supported in part by the Spanish Ministry of Science
and Innovation, through the “Severo Ochoa Programme for Centres of Excellence in R$\&$D (CEX2019-000904-S \& CEX2023-001347-S)” and 114703GB-100. We were also partially supported by the ERC Advanced Grant 788250, and by the SNF grant FLUTURA: Fluids, Turbulence, Advection No. 212573.

\bibliographystyle{alpha}

\end{document}